\documentclass[journal,twoside,web]{ieeecolor}
\usepackage{generic}
\usepackage{cite}
\usepackage{amsmath,amssymb,amsfonts}
\usepackage{algorithmic}
\usepackage{graphicx}
\usepackage{textcomp}
\usepackage{hyperref}
\usepackage{graphicx,color}
\graphicspath{{./IMG/}{images/}{./IMG/SAGEMATH_PLOTS/}}
\usepackage{mathrsfs}
\usepackage{subfigure}
\usepackage{url}
\usepackage{booktabs}
\usepackage{array}
\usepackage[table]{xcolor}
\usepackage{mathtools} 		
\usepackage{epstopdf}
\usepackage{cite}
\usepackage[vlined,ruled]{algorithm2e}
\usepackage{upgreek}
\usepackage{dsfont}
\usepackage{stackengine} 

\newtheorem{theorem}{Theorem}[section]
\newtheorem{lemma}[theorem]{Lemma}

\newtheorem{proposition}[theorem]{Proposition}
\newtheorem{problem}{Problem}
\newtheorem{remark}{Remark}
\newtheorem{assumption}{Assumption}

\newcommand{\mc}{\mathcal}

\newcommand{\real}{\mathbb{R}}

\newcommand{\realpos}{\mathbb{R}_{\geq 0}}

\newcommand{\tsp}{\mathsf{T}} 
 
\newcommand{\inv}{{\negat 1}} 
\newcommand{\negat}{\scalebox{0.75}[.9]{\( - \)}}

\newcommand*{\QEDB}{\hfill\ensuremath{\square}}
\newcommand*{\QEDBL}{\hfill\ensuremath{\blacksquare}}
\newcommand\oprocendsymbol{\hbox{$\square$}}
\newcommand\oprocend{\relax\ifmmode\else\unskip\hfill%
\fi\oprocendsymbol}

\newcommand{\map}[3]{#1: #2 \rightarrow #3}

\newcommand{\sbs}[2]{{#1}_{\textup{#2}}}
\newcommand{\sps}[2]{{#1}^{\textup{#2}}}

\newcommand{\norm}[1]{\Vert #1 \Vert}

\usepackage{epstopdf}
\def\BibTeX{{\rm B\kern-.05em{\sc i\kern-.025em b}\kern-.08em
    T\kern-.1667em\lower.7ex\hbox{E}\kern-.125emX}}
\begin{document}
\title{Time-Varying Optimization of LTI Systems 
via Projected Primal-Dual Gradient Flows}
\author{Gianluca Bianchin, Jorge Cort\'{e}s, Jorge I. Poveda,  and Emiliano Dall'Anese\thanks{G. Bianchin, J. I. Poveda, and E. Dall'Anese are with the Department of Electrical, Computer and Energy Engineering, University of Colorado Boulder. J. Cort\'{e}s is with the Department of Mechanical and Aerospace Engineering at the University of California San Diego. This work was supported by the National Science Foundation (NSF) through the Awards CMMI 2044946 and 2044900 and CRII: CNS-1947613, and  by  the  National  Renewable  Energy  Laboratory through the subcontract UGA-0-41026-148.} \hspace{-1cm}}

\maketitle

\begin{abstract}
This paper investigates the problem of regulating, at every time, a 
linear dynamical system to the solution trajectory of a time-varying 
constrained  convex  optimization problem.  
The proposed feedback controller is based on an adaptation of the 
saddle-flow dynamics, modified to take into account 
projections on constraint sets and output-feedback from the plant.
We derive sufficient conditions on the tunable parameters of the 
controller (inherently related to the time-scale separation between
plant and controller dynamics) to guarantee exponential   
input-to-state stability of the closed-loop system. 
The analysis is tailored to the case of time-varying strongly convex 
cost functions and polytopic output constraints. The theoretical 
results are further validated in a ramp metering control  problem in 
a network of traffic highways. 
\end{abstract}
%
%

\section{Introduction}
\label{sec:introduction}

\IEEEPARstart{T}{his} paper investigates the problem of online 
optimization of linear time-invariant (LTI) systems. 
The objective is to design an output feedback controller
to  steer the inputs and outputs of the system towards the 
solution trajectory of a time-varying optimization problem  
(see Fig. 1). Such problems correspond to scenarios where cost and 
constraints may change over time to reflect dynamic performance 
objectives or simply to take into account time-varying unknown 
disturbances entering the system. This setting emerges  in many 
engineering applications, including power systems 
\cite{MC-ED-AB:20,menta2018stability}, transportation networks 
\cite{GB-FP:19-tits,kutadinata2016enhancing}, and communication 
systems \cite{SL-DL:99}.

The design of  feedback controllers inspired from optimization algorithms has received significant attention during the last decade~\cite{Jokic2009controller,brunner2012feedback,lawrence2018optimal,lawrence2018linear,MC-ED-AB:20,menta2018stability,zheng2019implicit,hauswirth2020timescale,li2020optimal,bianchin2020online}. 
While most of the existing works focus on
the design of optimization-based controllers for static 
problems~\cite{Jokic2009controller,brunner2012feedback,lawrence2018optimal,lawrence2018linear,menta2018stability,hauswirth2020timescale,li2020optimal}, or consider unconstrained time-varying problems~\cite{zheng2019implicit,bianchin2020online}, an open research question is whether controllers can be synthesized  to track solutions trajectories of time-varying problems with input and output constraints. Towards this direction, in this paper we consider 
optimization problems with a time-varying strongly convex cost, 
time-varying linear constraints on the output, and convex constraints 
on the input.
We leverage online saddle-point dynamics for controller synthesis, 
and we establish the input-to-state stability \cite{ES-YW:95} property
for the system resulting from interconnecting the controller with the 
dynamical system.
In particular, we leverage tools from singular perturbation 
theory~\cite{HKK:96}  
to provide sufficient conditions on the tunable controller 
parameters  to guarantee  tracking of the optimal solution trajectory. 
We remark that, while \cite{qu2018exponential,Mihailo19,JC-SN:19,nguyen2018contraction,cisneros2020distributed} show that primal-dual dynamics
for have an exponential rate of  convergence, the main 
challenges here are to derive exponential 
stability results for problems that are time-varying and 
where  primal-dual dynamics are interconnected with a 
dynamical system subject to unknown disturbances 
(as in Fig.~\ref{fig:controllerBlocks}).

\begin{figure}[t]
\centering 
\includegraphics[width=.8\columnwidth]{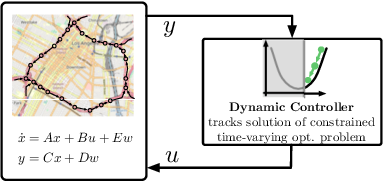}
\caption{Online saddle-flow optimizer used as an output 
feedback controller for LTI systems subject to unknown time-varying 
disturbances. $x$ denotes the system state, $u$ is the control input, 
$w$ denotes an unknown and unmeasurable disturbance, and $y$ is 
the system output.}
\vspace{-.6cm}
\label{fig:controllerBlocks}
\end{figure}

\textit{Related work}. 
In the case of static plants (i.e., where the dynamics of the 
system are infinitely fast), controllers conceptually-inspired from 
continuous-time saddle-point dynamics (or flows) are studied in 
\cite{Jokic2009controller} for optimization problems with 
time-invariant costs and constraints on the system outputs, whereas 
more general saddle-point flows are studied 
in~\cite{brunner2012feedback,li2020optimal,SaddlePointMonotone}, and
\cite[Sec. 3]{PovedaLinaAuto21}.
While the above works focus on optimization problems with static 
plants, the authors in
\cite{lawrence2018optimal,lawrence2018linear,hauswirth2020timescale,menta2018stability} prove that gradient-flow dynamics can be used 
as feedback controllers for dynamical systems in the case of 
unconstrained optimization problems with time-invariant costs. 
The work \cite{hauswirth2020timescale} also extends these results to
the case of constraints on the system inputs by using projected 
gradient flows. Constraints on the system outputs are considered in 
\cite{MC-ED-AB:20}, together with a controller inspired from 
primal-dual dynamics based on the Moreau envelope.
For time-varying unconstrained optimization 
problems, prediction-correction algorithms are used 
in~\cite{zheng2019implicit}. Exponential rates of convergence were 
proved for the first time in~\cite{bianchin2020online} for dynamic 
controllers based on gradient flows and accelerated hybrid dynamics.

In terms of classes of  plants, stable LTI systems are considered in
\cite{menta2018stability,bianchin2020online,MC-ED-AB:20}, stable 
nonlinear systems in~\cite{hauswirth2020timescale}, 
input-linearizable systems in~\cite{zheng2019implicit}, and 
input-affine nonlinear system in \cite{brunner2012feedback}.  
Finally, \cite{liao2020time,figura2020instant} 
consider online implementations of optimization problems arising in 
model predictive control.

\textit{Contributions}. This work features three main contributions.
C1) We design an output feedback controller, inspired from 
primal-dual dynamics, to regulate a dynamical system to 
the solution trajectory of a time-varying constrained optimization 
problem without requiring information or measurements of the 
external disturbances entering the state equation. For problems 
with equality constraints, the controller is designed 
based on the classical Lagrangian function. Instead, for problems 
with inequality constraints, we employ a regularized 
Lagrangian~\cite{JK-AN-UVS:11} to guarantee exponential convergence 
to an approximate KKT trajectory.  
C2) We consider constraints on the system input and we propose a 
novel projected primal-dual feedback controller that guarantees 
constraint satisfaction. 
Differently from using the classical projection on the 
tangent cone, the proposed controller yields trajectories that are
continuously differentiable, which allows us to simplify the 
analysis and to establish strong robustness guarantees. 
As a minor contribution, we demonstrate that the proposed framework 
is applicable to more-general LTI systems, including switched 
systems with common quadratic Lyapunov functions. 
C3) We apply the proposed controllers to solve a ramp metering problem 
in traffic systems. We compare our results with state-of-the-art 
controllers, including ALINEA \cite{MP-AK:02} and 
model predictive control, illustrating the advantages of our method. 

We emphasize that, relative to~\cite{MC-ED-AB:20}: (i) our sufficient  
conditions for convergence are easier to check as they do not require 
to numerically solve a linear matrix inequality, and (ii)  our 
framework does not require to compute the Moreau envelope. Relative 
to~\cite{Jokic2009controller,li2020optimal,brunner2012feedback,hauswirth2020timescale}, we account for time variability in the cost 
functions and in the disturbances, and we prove  exponential 
convergence.  Relative 
to~\cite{cisneros2020distributed,edwards2000input}, we investigate 
saddle-point dynamics when coupled with a dynamical system. 

\textit{Organization.}
We present in Section~\ref{sec:problem} our problem formulation. 
Section~\ref{sec_projected_systems} 
develops a projected primal-dual output feedback controller for 
problems with input constraints and output inequality constraints.
Section~\ref{sec:no_projection} considers problems with output 
equality constraints.
Section~\ref{sec_application_metering} presents numerical results by
focusing on a ramp metering problem in traffic systems.
Finally, Section~\ref{sec_conclusions} summarizes our conclusions.

\textit{Notation.}
Given vectors $x\in \real^n$ and $u\in\real^m$, we let $(x,u) \in \real^{n+m}$ denote their concatenation. We use $\bar{\lambda}(M)$ and $\underline{\lambda}(M)$ to denote the  largest and smallest eigenvalues of the symmetric matrix $M$, respectively. Finally,  
$P_\Omega:\mathbb{R}^\sigma \to \mathbb{R}^\sigma$ denotes the 
Euclidean projection of $z$ onto a closed convex set
$\Omega \subseteq \real^\sigma$, namely $P_\Omega (z) := \arg \min_{v \in \Omega} \norm{z-v}.$

\section{Problem Formulation}
\label{sec:problem}
We consider LTI dynamical systems described by:
\begin{equation}
\begin{split}
\label{eq:plantModel}
\dot x &=  A x + B u + E w_t, \\
y &= C x + D w_t,
\end{split}
\end{equation}
where $\map{x}{\realpos}{\real^n}$ is the state, 
$\map{u}{\realpos}{\real^m}$ is the input, 
$\map{y}{\realpos}{\real^p}$ is the output,
and $\map{w_t}{\realpos}{\real^q}$ is an unknown 
and time-varying exogenous input or disturbance (the 
notation $w_t$ emphasizes the dependence on time).
We make the following stability assumption on the plant.


%
\begin{assumption}
\label{ass:stabilityPlant}
The matrix $A$ is Hurwitz stable, namely, for any 
$Q_x \in \real^{n\times n}, Q_x \succ 0$, there exists 
$P_x \in \real^{n \times n}, P \succ 0,$ such that 
$A^\tsp P_x+P_xA = -Q_x$.
\QEDB
\end{assumption}
\vspace{.1cm}

Under Assumption \ref{ass:stabilityPlant}, for fixed vectors 
$\sbs{u}{eq} \in \real^m$, $\sbs{w}{eq} \in \real^q$, 
\eqref{eq:plantModel} has a unique stable equilibrium point 
$\sbs{x}{eq} =-A^{-1}(B \sbs{u}{eq}+E \sbs{w}{eq})$. 
Moreover, at equilibrium, the relationship between system inputs and outputs is given by the algebraic relationship:
\begin{align}
\label{eq:yTransferFunctions}
\sbs{y}{eq} = \underbrace{-C A^\inv B}_{:=G} \sbs{u}{eq} 
+ \underbrace{(D -C A^\inv E)}_{:=H} \sbs{w}{eq}.
\end{align}
Given any time-varying and unknown exogenous input $w_t$ to 
\eqref{eq:plantModel}, we focus on the problem of regulating the 
plant to the solutions of the following time-varying optimization 
problem:
\vspace{-.1cm}
\begin{subequations}
\label{opt:objectiveproblem}
\begin{align}
\label{opt:objectiveproblem-a}
(u_t^*, y_t^*) \in  
\underset{\bar u \in \mc U, ~ \bar y \in \real^p}{\arg \min}  ~~ & 
\phi_t (\bar u) + \psi_t (\bar y)\\
\label{opt:objectiveproblem-b}
\text{s.t.} ~~~~ & \bar y = G  \bar u  +H w_t\\
\label{opt:objectiveproblem-c}
& K_t \;  \bar y \leq e_t, 
\end{align}
\end{subequations}
where for all $t \in \realpos$, 
$\map{\phi_t}{\real^m}{\real}$,
$\map{\psi_t}{\real^p}{\real}$. Moreover, the maps
$t \mapsto K_t \in \real^{r \times p}$ and
$t \mapsto e_t \in \real^r$ describe a time-varying  output 
constraint, while $\mc U \subseteq \real^m$ denotes a closed and convex set 
describing constraints on the input.
Problem \eqref{opt:objectiveproblem} formalizes a regulation 
problem, 
where the objective is to select an optimal input-output pair 
$(u^*_t, y^*_t)$ that minimizes the cost specified by the 
loss functions $\phi_t$ and $\psi_t$.  
We note that, because cost functions and constraints are 
time-varying, the solutions of \eqref{opt:objectiveproblem} 
are also time-varying, and thus they characterize optimal 
trajectories. We impose the following regularity assumptions on the 
temporal evolution of~\eqref{opt:objectiveproblem}.

\vspace{.1cm}

\begin{assumption}
\label{ass:lipschitzAndConvexity}
The following properties hold.
\begin{enumerate}
\item[(a)] The functions $u\mapsto \phi_t(u)$  and  
$y\mapsto \psi_t(y)$ are continuously differentiable, uniformly in $t \in \realpos$.
\item[(b)] The function $u \mapsto \phi_t(u)$ is $\mu_u$-strongly
convex, uniformly in $t \in \realpos$.
\item[(c)] There exist $\ell_u,\ell_y >0$ such that for 
every $u, u' \in \real^m$ and $y, y' \in \real^p$, 
$\norm{\nabla \phi_t(u) - \nabla \phi_t(u')} 
\leq \ell_u \norm{u-u'}$, 
$\norm{\nabla \psi_t(y) - \nabla \psi_t(y')} \!\!\leq\!\! \;\ell_y \norm{y-y'}$,
uniformly in $t \in \realpos$. 
\item[(d)]
For all $u \in \real^m$, $y \in \real^p$,
$t\mapsto \nabla \phi_t(u)$ and $t\mapsto \nabla \psi_t(y)$ 
are locally Lipschitz. 
\QEDB
\end{enumerate}
\end{assumption}
\begin{assumption}
\label{ass:slaterConditionAndRank}
Problem~\eqref{opt:objectiveproblem} is feasible, and Slater's 
condition \cite[Assumption 1]{JK-AN-UVS:11} holds for each $t\in \realpos$.  \QEDB
\end{assumption}
\vspace{.1cm}

\begin{assumption}
\label{ass:timeRegularity}
The following regularity properties~hold.
\begin{enumerate}
\item[(a)]  $t \mapsto w_t$ is locally absolutely continuous. 
\item[(b)]  The functions $t \mapsto [K_t]_{ij}$ and $t \mapsto [e_t]_{i}$
$i = 1, \dots, r$, $j = 1, \ldots, p$, are locally Lipschitz, and 
there exists 
$\overline{K} \in \realpos$, $\bar{e} \in \realpos$, such that 
$\norm{K_t} < \bar K$ and $\norm{e_t} < \bar e$.
\QEDB
\end{enumerate}
\end{assumption}
\vspace{.1cm}

Under Assumptions 
\ref{ass:lipschitzAndConvexity}--\ref{ass:slaterConditionAndRank}, the 
minimizer $(u^*_t, y^*_t)$ of 
\eqref{opt:objectiveproblem} is unique  for every $t \in \realpos$
\cite[Page 2]{JK-AN-UVS:11}, while Assumption \ref{ass:timeRegularity} 
guarantees that inputs and constraints of 
\eqref{opt:objectiveproblem} vary continuously in time. 
The problem focus of this work is formalized next.

\vspace{0.1cm}
\begin{problem}
\label{prob:1}
Design a dynamic output-feedback controller for~\eqref{eq:plantModel}
such that the inputs and outputs of \eqref{eq:plantModel} converge 
exponentially to the time-varying optimizer of 
\eqref{opt:objectiveproblem}, up to an asymptotic error that accounts 
for the temporal variability of both the optimizer and of the unknown 
disturbance. \QEDB
\end{problem}
\vspace{0.1cm}

\section{Closed-loop Projected Saddle-Point Flows}
\label{sec_projected_systems}
In this section, we present our controller synthesis method 
and we establish explicit convergence error bounds.

\subsection{Controller Synthesis}
For controller synthesis, we employ a regularized Lagrangian 
function and we use a controller structure that relies on a 
modification of 
the saddle-point flow dynamics \cite{Jokic2009controller}.
Consider the following Lagrangian function 
for \eqref{opt:objectiveproblem}: 
\begin{align*}
\mc L_{t}(u,\lambda) &:= \phi_t(u) + \psi_t(Gu+Hw_t) \\ 
&\quad\quad\quad\quad\quad\quad\quad
+ \lambda^\tsp(K_t(Gu+Hw_t)-e_t),    
\end{align*}
where 
$\lambda \in \realpos^r$ denotes the vector of dual variables. 
We define the regularized Lagrangian function as follows:
\begin{align}
\label{eq:LagrangianInequalityAugmented}
\mc L_{\nu,t}(u,\lambda) := &\mc L_{t}(u,\lambda) 
- \frac{\nu}{2} \norm{\lambda}^2,
\end{align}
where $\nu \in \real_{>0}$. The regularization term 
$- \frac{\nu}{2} \norm{\lambda}^2$ has the effect of making the function
$\mc L_{\nu,t}(u,\lambda)$ strongly concave in $\lambda$, for any 
$u \in \real^m$ (see~\cite{JK-AN-UVS:11}). 
As a result, the regularization term induces a saddle-point map that is
strongly monotone, uniformly in time \cite{JK-AN-UVS:11}. 
On the other hand, the us of a regularization term comes at 
the cost of perturbing the saddle points. To this aim, we let
\begin{align}
\label{eq:defSaddlePoint}
z_t^* &:= (u_t^*, \lambda_t^*), & 
z_{\nu,t}^* :&= (u_{\nu,t}^*, \lambda_{\nu,t}^*),
\end{align}
denote any saddle-point of $\mc L_{t}(u,\lambda)$ and  the  saddle 
point of $\mc L_{\nu,t}(u,\lambda)$, respectively. We quantify the 
error due to regularization in the following result (adapted 
from~\cite[Prop.  3.1]{JK-AN-UVS:11}).
\begin{lemma}
\label{lem:regularizationError}
Let Assumptions 
\ref{ass:lipschitzAndConvexity}-\ref{ass:slaterConditionAndRank} 
hold. For each $t \in \realpos$, the following bound holds:
\begin{align}
\label{eq:shiftOptimizer}
\mu_u \norm{u_{\nu,t}^* - u_t^*}^2 + \frac{\nu}{2} \norm{\lambda_{\nu,t}^*}^2 \leq \frac{\nu}{2} \norm{\lambda_t^*}^2,
\end{align}
where $(u_t^*, \lambda_t^*)$ and 
$(u_{\nu,t}^*, \lambda_{\nu,t}^*)$ are as in \eqref{eq:defSaddlePoint}.
In particular, inequality \eqref{eq:shiftOptimizer} implies 
that 
$\norm{u_{\nu,t}^* - u_t^*} \leq \sqrt{\frac{\nu}{2 \mu_u}} \norm{\lambda^*_t}$. 
\end{lemma}

\begin{remark}
Lemma \ref{lem:regularizationError} shows that the error induced by 
the regularization term is bounded by the norm of the optimal 
multipliers of the non-regularized problem. Consequently, when the 
optimal solution is strictly inside the feasible set, then 
$\lambda_t^*=0$ and  the solution $u_{\nu,t}^*$ coincides with 
$u_t^*$.  
\QEDB
\end{remark}
\smallskip


For controller synthesis we define the following functions, which can 
be interpreted as modified gradients of 
\eqref{eq:LagrangianInequalityAugmented}:
\begin{subequations}
\label{eq:LuProjection}
\begin{align}
\label{eq:LuProjection-a}
L_{u,t}(u,y,\lambda) & := \nabla \phi_t(u) + G^\tsp \nabla \psi_t(y) 
+ G^\tsp K_t^\tsp \lambda,\\
\label{eq:LuProjection-b}
L_{\lambda,t}(y,\lambda) & := K_t y - e_t - \nu \lambda,
\end{align}
\end{subequations}
where we note that, with respect to the gradients of $\mc L_{\nu,t}$,
in $L_{u,t}$ and $L_{\lambda,t}$ the map $Gu+Hw_t$
has been replaced by variable $y$. 
Using~\eqref{eq:LuProjection}, we propose the following \emph{online 
projected primal-dual controller} applied to \eqref{eq:plantModel} 
(see Fig.~\ref{fig:controllerBlocks}):
\vspace{-.5em}
\begin{subequations}
\label{eq:primal-DualProjection}
\begin{align}
\label{eq:primal-DualProjection-a}
\varepsilon \dot x &= A x + B u + E w_t, 
\quad\quad\quad\quad\quad y = C x + D w_t,\\    
\label{eq:primal-DualProjection-b}
\dot u &=  P_\mc U \big(u- \eta L_{u,t}(u,y,\lambda)\big) - u,\\
\label{eq:primal-DualProjection-c}
\dot \lambda &= P_{\mathcal{C}} \big(\lambda + 
\eta L_{\lambda,t}(y,\lambda)\big) - \lambda,
\end{align}
\end{subequations}
where $\varepsilon, \eta >0$ are plant and controller gains that  induce a time-scale separation between the plant and the controller, 
$v \mapsto P_{\Omega}(v)$ denotes the Euclidean projection onto the closed and 
convex set $\Omega$, and $\mc C:=\mathbb{R}_{\geq0}^r$. Three
important observations on 
\eqref{eq:primal-DualProjection-b}-\eqref{eq:primal-DualProjection-c} 
are in order.
First, the structure of the controller is inspired by first-order
optimization methods, where the algebraic map $Gu+Hw_t$ has been 
replaced by measurements of the output $y$ (thus making the algorithm 
``online''). Second, the controller
does not require any knowledge regarding the exogenous disturbance 
$w_t$. Third, even when the LTI system and the saddle-flow 
dynamics are stable (in open-loop), the interconnection
\eqref{eq:primal-DualProjection}  is not guaranteed to be stable without further conditions on the controller parameters~\cite{HKK:96}.

\begin{remark} 
The choice of dualizing the constraint 
$K_t  (Gu+Hw_t) \leq e_t$ allows us to naturally enforce 
constraints that are time-varying and parametrized by
the unknown vector~$w_t$. This is because the steady-state 
relationship $Gu+Hw_t$ is replaced by instantaneous feedback 
$y_t$ in~\eqref{eq:LuProjection-b}. The alternative route of combining 
the constraint $K_t (Gu+Hw_t) \leq e_t$ with the convex constraint 
$u \in \mc U$ and recast both of them as a convex constraint of the 
form  $u \in \mc U \cap \{u: K_t (Gu+Hw_t) \leq e_t\}$ would 
result in an unknown constraint set, thus making the computation
of the projection not possible. 
\QEDB
\end{remark}
\smallskip

\begin{remark}
Given a closed convex set $\Omega \subseteq \real^\sigma$ and a vector
field $\map{F}{\Omega}{\real^\sigma}$, the standard projected 
dynamical system \cite{AN-DZ:12} associated with $F(v)$ is given by:
\begin{align}
\label{eq:standardProjectedDynamics}
\dot v &= \lim_{\delta \rightarrow 0^+} 
\frac{P_\Omega(v + \delta F(v))-v}{\delta}.
\end{align}
We note that, in general, \eqref{eq:standardProjectedDynamics} is a 
discontinuous dynamical system. On the contrary, the vector field in
\eqref{eq:primal-DualProjection-b}-\eqref{eq:primal-DualProjection-c} 
is Lipschitz continuous. For static optimization problems, similar 
dynamics have been studied in e.g. 
\cite{gao2003exponential,YSX-JW:00}. However, to the best of our knowledge, 
\eqref{eq:primal-DualProjection-b}-\eqref{eq:primal-DualProjection-c} 
is the first projected output feedback controller with 
Lipschitz-continuous vector fields.  \QEDB
\end{remark}

\vspace{0.1cm}
Fig.~\ref{fig:projectionIllustration} provides a representative 
example of the trajectories produced by the considered projected 
output feedback controllers, and compares them with those generated by 
a controller with a discontinuous projection of the form \eqref{eq:standardProjectedDynamics}.

\begin{figure}[t]
\centering 
\includegraphics[width=.7\columnwidth]{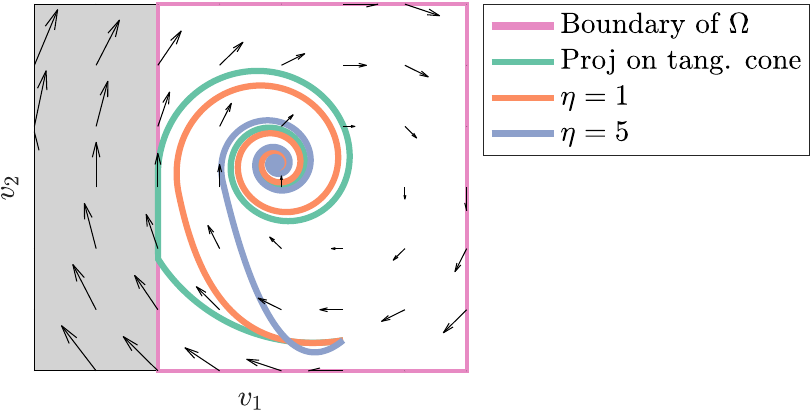}
\caption{Comparison between trajectories of 
\eqref{eq:standardProjectedDynamics} and of the smooth 
projection \eqref{eq:primal-DualProjection} for a 2-D vector field. 
Black arrows show the vector field.}
\label{fig:projectionIllustration}
\vspace{-.5cm}
\end{figure}
%

\subsection{Stability and Tracking Analysis}
\label{sec:trackingProjected}
In this section we characterize the transient behavior 
of~\eqref{eq:primal-DualProjection}. 
To this aim, in what follows we use the notation:
\begin{align}
\label{eq:z-zNuTilde}
z &:= (u, \lambda), &
\tilde{z}_\nu& := z-z_{\nu,t}^*,
\end{align}
to denote the joint controller state and the controller tracking 
error, where $z_{\nu,t}^*$ is as in \eqref{eq:defSaddlePoint}. 
Similarly, we use
\begin{align}
\label{eq:xiNu-xiNuStar}
\xi &:= (x, z), &
\xi^*_{\nu,t} &:= (x^*_{\nu,t}, z^*_{\nu,t}), &
\tilde{\xi}_\nu &:=\xi-\xi_{\nu,t}^*,
\end{align}
to denote the joint state of \eqref{eq:primal-DualProjection}, the 
saddle-point of \eqref{eq:LagrangianInequalityAugmented}, with 
$x_{\nu,t}^*=-A^\inv(Bu_{\nu,t}^*+Hw_t)$, and the 
tracking error, respectively.
We begin by characterizing the existence of solutions.

\smallskip
\begin{lemma}
\label{lem:differentiableSolutions}
Let Assumptions~\ref{ass:lipschitzAndConvexity}--\ref{ass:timeRegularity} 
hold. 
For each $\xi_0=(x_0,u_0,\lambda_0)\in \real^{n+m+r}$, there exists a
unique solution $\xi(t)$ of \eqref{eq:primal-DualProjection} with 
$\xi(0) = \xi_0$.
Moreover, $\xi$ is continuously differentiable and it is
maximal, i.e., it is defined for all $t \in \realpos$.
\end{lemma}
\begin{proof}
This claim follows from the following facts: (i) the 
projection mapping is globally 
Lipschitz~\cite{gao2003exponential,YSX-JW:00}, (ii) under 
Assumptions~\ref{ass:lipschitzAndConvexity}--\ref{ass:timeRegularity}, the maps $L_{u,t}(u,y,\lambda)$ and 
$L_{\lambda,t}(y,\lambda)$  are globally Lipschitz in 
$(u,y,\lambda)$ uniformly in $t$, and locally Lipschitz 
with respect to $t$, (iii) the composition of globally 
Lipschitz functions is globally Lipschitz, and (iv) under 
Assumption~\ref{ass:timeRegularity}(a) the plant  dynamics 
are locally Lipschitz in $t$.~
\end{proof}
\vspace{.1cm}
Lemma \ref{lem:differentiableSolutions} guarantees that the 
trajectories of \eqref{eq:primal-DualProjection} are continuously 
differentiable (see Fig. 
\ref{fig:projectionIllustration}). 
Moreover, since trajectories are maximal, Lemma 
\ref{lem:differentiableSolutions} guarantees that trajectories
have no finite escape time.  
The latter property is leveraged to prove
the following result, which  establishes attractivity and forward invariance of the feasible set (see \cite[Thm 3.2]{YSX-JW:00}). 
\vspace{.1cm}

\begin{lemma}
\label{lem:differentiableSolutionsProjection}
Let Assumptions~\ref{ass:lipschitzAndConvexity}--\ref{ass:timeRegularity} hold.
If $u(t_0) \not \in \mc U$ (resp. $\lambda(t_0) \not \in \mc C$)
for some $t_0 \in \realpos$, then the trajectory $u(t)$ (resp. 
$\lambda(t)$) approaches exponentially the set $\mc U$ (resp. the set 
$\mc C$) for $t > t_0$. If $u(t_0) \in \mc U$ (resp. 
$\lambda(t_0) \in \mc C$) for some $t_0 \in \realpos$, then 
$u(t) \in \mc U$ (resp. $\lambda(t) \in \mc C$) for all $t \geq t_0$. 
\end{lemma}
\vspace{.1cm}

\begin{remark}
Lemma \ref{lem:differentiableSolutionsProjection} guarantees that, if 
$u(t_0) \in \mc U$, then the constraint $u(t) \in \mc U$ is 
satisfied for all $t \geq t_0$. In contrast, because the constraint
\eqref{opt:objectiveproblem-c} is dualized in 
\eqref{eq:LagrangianInequalityAugmented}, the inequality 
$K_t y(t) \leq e_t$ is guaranteed to hold only asymptotically, even 
when $K_t y(t_0) \leq e_t$ for some $t_0 \in \realpos$.~
\QEDB
\end{remark}
\vspace{.1cm}

The following lemma establishes a relationship between the 
saddle-point of the regularized Lagrangian 
\eqref{eq:LagrangianInequalityAugmented} and the equilibria  
of \eqref{eq:primal-DualProjection}. 
The proof is omitted due to space limitations.
%


\vspace{.1cm}
\begin{lemma}
\label{prop:equilibriaProjectedSystem}
For any $t \in \realpos$ and for any fixed $w_t \in \real^q$, let 
$\sbs{\xi}{eq} := (\sbs{x}{eq},\sbs{u}{eq},\sbs{\lambda}{eq})$
denote an equilibrium of \eqref{eq:primal-DualProjection}.
If Assumptions \ref{ass:stabilityPlant}-\ref{ass:timeRegularity} hold, 
then $\sbs{\xi}{eq}$ is unique and it coincides with the unique 
saddle-point of~\eqref{eq:LagrangianInequalityAugmented}, as defined by
\eqref{eq:xiNu-xiNuStar}.
\end{lemma}
\vspace{.1cm}

To characterize the transient behavior of
\eqref{eq:primal-DualProjection}, we first show that, when the 
dynamics of the plant \eqref{eq:plantModel} are infinitely fast, 
the controller 
\eqref{eq:primal-DualProjection-b}-\eqref{eq:primal-DualProjection-c} 
converges exponentially to the saddle-point of the 
regularized Lagrangian, modulo an asymptotic error that depends
on the time-variability of the optimizer $z^*_t$.~
\vspace{.1cm}
\begin{proposition}
\label{prop:exponentialProjectedPrimalDualDisturbance}
Let Assumptions 
\ref{ass:stabilityPlant}-\ref{ass:timeRegularity} hold, let
$\mu := \min\{\mu_u, \nu\}$, 
$\ell =: \sqrt{2} (\overline{K} + \max \{ \ell_u + \norm{G}^2 \ell_y,
\nu \})$.
If $\varepsilon=0$ and the controller gain satisfies 
$\eta < \frac{4 \mu}{\ell^2}$, then for any $t_0 \in \realpos$:
\begin{align}
\label{eq:zBoundProjection} 
\norm{\tilde{z}_\nu(t)} \leq e^{-\frac{1}{2} \rho_z (t-t_0)} 
 \norm{\tilde{z}_\nu(t_0)} 
+ \frac{2}{\rho_z} \text{ess} \sup_{\tau \geq t_0} \norm{\dot z^*_{\nu,\tau}},
\end{align}
for all $t \geq t_0$,
where $\rho_z = \eta (\mu - \frac{\eta \ell^2}{4})$, and 
$\tilde z_\nu$ denotes the controller tracking error as in 
\eqref{eq:z-zNuTilde}.
\end{proposition}
\vspace{.1cm}

The proof of this result is postponed to
Appendix~\ref{sec:proofsProjected}. 
Proposition \ref{prop:exponentialProjectedPrimalDualDisturbance}
guarantees that \eqref{eq:primal-DualProjection} is input-to-state 
stable \cite{ES-YW:95} with respect to the time derivative of the 
optimizer (here, $\dot z^*_{\nu,\tau}$ denotes the distributional 
derivative \cite{GG:08} of $z^*_{\nu,\tau}$, see Remark 
\ref{rem:absoluteContinuity}).
Notice that the rate of convergence $\rho_z$ can be tuned by properly 
tuning the  controller gain $\eta$.

\begin{remark}
\label{rem:absoluteContinuity}
We note that, under Assumptions 
\ref{ass:stabilityPlant}--\ref{ass:timeRegularity}, the saddle-point 
trajectory $t\mapsto z_{\nu,t}^*$ is locally Lipschitz and hence 
absolutely continuous on compact sets. Thus, the essential supremum 
of $\norm{\dot z^*_{\nu,\tau}}$ is well defined.
To see this, notice that, $z^*_{\nu,\tau}$ solves the following 
Variational Inequality:

\vspace{-.3cm}
\begin{small}
\begin{align*}
&(u - u_{\nu,t}^*) (\nabla \psi_t(u_{\nu,t}^*) 
+ G^\tsp \nabla \phi_t(G u_{\nu,t}^* + H w_t) 
+ G^\tsp K_t^\tsp \lambda_{\nu,t}^*) \geq 0,\\
&(\lambda - \lambda_{\nu,t}^*) (K_t (G u_{\nu,t}^* + H w_t) - e_t - \nu \lambda_{\nu,t}^*) \geq 0,
\end{align*}
\end{small}
\vspace{-.35cm}

\noindent
which holds for all $u \in \mc U$, $\lambda \in \mc C$, and for all 
$t \in \realpos$. It follows from Assumptions 
\ref{ass:stabilityPlant}-\ref{ass:timeRegularity} and from our 
regularization method \eqref{eq:LagrangianInequalityAugmented} that 
the mapping defining the above variational inequality is locally 
Lipschitz in $(u, \lambda)$, and thus \cite[Cor. 2B.3]{AD-TR:09} 
guarantees that $z^*_{\nu,\tau}$ is locally Lipschitz. 
Hence, by Rademacher's theorem \cite[Thm. 23.2]{ED:02}, 
$t \mapsto z^*_{\nu,\tau}$ is differentiable almost everywhere (a.e.).
\QEDB
\end{remark}
\smallskip

Next, we provide a sufficient condition on the time-scale separation 
between  the plant \eqref{eq:primal-DualProjection-a} and the
feedback  controller 
\eqref{eq:primal-DualProjection-b}-\eqref{eq:primal-DualProjection-c} 
to ensure convergence to the optimal~trajectory.

\smallskip
\begin{theorem}
\label{thm:exponentialInequalityConstr}
Let Assumptions 
\ref{ass:stabilityPlant}-\ref{ass:timeRegularity} hold, let
$\ell =: \sqrt{2} (\overline{K} + \max \{ \ell_u + \norm{G}^2 \ell_y,
\nu \})$ and $\mu := \min\{\mu_u, \nu\}$. If 
\begin{align}
\label{eq:boundEpsilonInequality}
\eta  < \frac{4 \mu}{\ell^2} ~~~\text{ and }~~~
\varepsilon < 
\frac{\rho_z \underline \lambda(Q_x)}{
4 \eta \norm{P_x A^\inv B} \Psi},
\end{align}
where $\rho_z = \eta (\mu - \frac{\eta \ell^2}{4})$, $\Psi =  \rho_z \ell_y \norm{C} \norm{G}  
+ \sqrt{2} \norm{C} (\ell_y \norm{G} + \bar{K}) k_0$,
$k_0=\max\{2 + \eta(\ell_u + \ell_y \norm{G}^2),  \norm{G} \bar{K} \}$,
and $P_x, Q_x$ are as in Assumption \ref{ass:stabilityPlant},
then for any $t_0 \in \realpos$:
\begin{align}
\label{eq:xiBoundProjection}
\norm{\tilde{\xi}_\nu(t)} &\leq 
\sqrt{\kappa} \norm{\tilde{\xi}_\nu(t_0)}  e^{-\frac{1}{2} \rho_\xi (t-t_0)}+ \frac{2}{\rho_z} \text{ess} \sup_{\tau \geq t_0} \norm{\dot z^*_{\nu,\tau}} \nonumber\\
& ~~~~~~+ \frac{4 \varepsilon \norm{PA^\inv E}}{\underline \lambda(Q_x)} 
\text{ess} \sup_{\tau \geq t_0} \norm{\dot w_\tau},
\end{align}
for all $t \geq t_0$, where 
$ \rho_\xi = \frac{1}{2} \min \left\{
2\rho_z, 
\frac{1}{4 \varepsilon} 
\frac{\underline \lambda(Q_x)}{\bar \lambda(P_x)}\right\}$, 
$\kappa = \max \{\frac{1}{2}, \bar \lambda (P_x)\}/
\min \{\frac{1}{2}, \underline \lambda (P_x)\}$, 
and $\tilde \xi_\nu$ is as in \eqref{eq:xiNu-xiNuStar}.
\end{theorem}
\smallskip
The proof of this result is presented in
Appendix~\ref{sec:proofsProjected}. 
Theorem~\ref{thm:exponentialInequalityConstr} shows that, under a 
sufficient separation between the time scales of the plant and of the 
controller, the trajectories of 
\eqref{eq:primal-DualProjection} globally exponentially converge to 
$\xi_{\nu,t}^*$ (which we recall is the trajectory of the unique 
saddle-point  of the regularized Lagrangian), modulo an 
asymptotic error that 
depends on the time-variability of the optimizer and of the exogenous 
disturbance.
Precisely, Theorem \ref{thm:exponentialInequalityConstr}
guarantees that \eqref{eq:primal-DualProjection} is input-to-state 
stable \cite{ES-YW:95} with respect to $\dot z^*_{\nu,\tau}$ and 
$\dot w_\tau$, where $\dot w_\tau$ denotes the distributional 
derivative \cite{GG:08} of $w_\tau$ (notice that, under Assumption 
\ref{ass:timeRegularity}(a), $\tau \mapsto w_\tau$ is differentiable 
a.e.).

Two important observations are in order. 
First, the upper bound for $\varepsilon$ is an 
increasing function of $\underline \lambda(Q_x)$ and $\rho_z$, that are
interpreted as the convergence rate of the open-loop plant and of the 
controller with $\varepsilon=0$, respectively.
Moreover, the bound is a decreasing function of $\norm{P_x A^\inv B}$.
Since $\norm{A^\inv} \rightarrow 0$ when the eigenvalues of $A$ are 
approaching the open right complex plane, the latter 
term takes into account the margin of stability of the open-loop plant. Second, we note that the rate of convergence $\rho_\xi$ is governed by 
the quantities $\rho_z$ and $\varepsilon$ (as well as matrices $P_x$ 
and $Q_x$), which are interpreted as the rate of convergence of the 
controller with $\varepsilon=0$ and the rate of convergence of the 
open-loop plant, respectively.
\vspace{.1cm}
\begin{remark}
The bound \eqref{eq:xiBoundProjection} depends on two main 
quantities: $\text{ess} \sup_{\tau \geq t_0} \norm{\dot z^*_{\nu,\tau}}$, 
which captures  the time-variability of $z^*_{\nu,t}$, and 
$\text{ess} \sup_{\tau \geq t_0} \norm{\dot w_\tau}$, which captures 
the shift in the equilibrium of \eqref{eq:plantModel} induced by the 
time-varying exogenous input $w_t$. 
Notably, when the optimization problem \eqref{opt:objectiveproblem} 
is time-invariant and $w_t$ is constant, \eqref{eq:xiBoundProjection} 
simplifies to an exponential stability result, of the form 
$\norm{\tilde{\xi}_\nu(t)} \leq 
\sqrt{\kappa} \norm{\tilde{\xi}_\nu(t_0)}  e^{-\frac{1}{2} \rho_\xi (t-t_0)}$.
\QEDB
\end{remark}

\subsection{Extensions}
\label{sec:swithching}
Our analysis suggests that the results can be extended in different 
directions. Here, we discuss two possible extensions. 
%

\vspace{0.1cm}
\subsubsection{Switched LTI Plants with Common Quadratic Lyapunov Functions} Theorem \ref{thm:exponentialInequalityConstr} can be 
extended to consider switched LTI plants 
of the form:
\begin{equation}
\begin{split}
\label{eq:plantModelSwitched}
\dot x &=  A_{\sigma} x + B_{\sigma} u + E_{\sigma} w_t, \\
y &= C_\sigma x + D_\sigma w_t,
\end{split}
\end{equation}
where $\sigma: \mathbb{R}_{\geq0}\to \mc Q$ is a switching signal 
taking values in the finite set $\mc Q$.
When all modes of \eqref{eq:plantModelSwitched} have a common equilibrium point $x_{\textrm{eq}}^*=A_{\sigma}^{-1}B_\sigma u+A_{\sigma}^{-1}E_{\sigma}w_t$ for all values of $\sigma$ and admit a common quadratic Lyapunov function $V$, the same construction for the Lyapunov function \eqref{eq:Uaux2} can be used to establish exponential ISS of the closed-loop system. Since in this case $G$ and $H$ in \eqref{eq:yTransferFunctions} are also common across the modes, the bounds in Theorem \ref{thm:exponentialInequalityConstr} still hold unchanged. 
This scenario emerges in applications where mode-dependent inner 
feedback controllers are implemented to stabilize each mode of the 
plant (so that all modes share a common equilibrium 
point~\cite{JPH-ASM:99}), but different controllers lead to different 
closed-loop transient performance. Note, however, that having a stable 
autonomous switched LTI system does not necessarily imply the existence
of a common quadratic Lyapunov function. Instead, it implies the 
existence of a common Lyapunov function that is homogeneous of degree 
2, e.g., piece-wise quadratic \cite{MartinPiceWise}. 
When matrices $C_{\sigma}$ and $D_{\sigma}$ are mode-dependent, Theorem~\ref{thm:exponentialInequalityConstr} can also be extended, 
provided that the pair $(G,H)$ remains common across modes and that 
\eqref{eq:boundEpsilonInequality} and 
\eqref{eq:xiBoundProjection} are modified to account for the 
worse-case bound among all~modes.

\subsubsection{Switched Plants with Average Dwell-Time Constraints} When the switched system \eqref{eq:plantModelSwitched} does not admit a
common Lyapunov function, it is still possible to obtain a result of 
the form \eqref{eq:xiBoundProjection}, provided the switching is slow 
``on the average''. In particular, if the switching signal $\sigma$ satisfies an average dwell-time constrain of the form
\begin{equation}
N_{\sigma}(t,\tau)\leq \eta_0(t-\tau)+N_0,    
\end{equation}
where $N_{\sigma}(t,\tau)$ denotes the number of discontinuities of 
$\sigma$ in the open interval $(\tau,t)$, $\eta_0 \in \real_{>0}$ 
denotes the switching signal dwell-time, and $N_0\geq 0$ is a chatter 
bound that guarantees that the number of consecutive switches is  finite at every time. In this case, it is possible to choose the controller gain $\eta$ sufficiently small such that the exponential stability property of the switched system is preserved, and the same construction \eqref{eq:Uaux2} carries over. This observation follows directly from the Lyapunov construction presented in \cite{bianchin2020online}, which permits the derivation of a result similar to Proposition 3.8 using quadratic Lyapunov functions. Characterizations of the conditions that emerge between $\eta$ and the time-scale separation parameters $(\varepsilon,\eta)$ can also be explicitly derived as in \cite{bianchin2020online}. However, unlike the results of \cite{bianchin2020online}, the results of this paper allow to consider online optimization problems with constraints.  To the best our knowledge, similar results for online optimization with constraints of switched systems have not been studied before.
%
\section{Online Primal-Dual Gradient Flow}
\label{sec:no_projection} 
In this section, we consider the problem of regulating 
\eqref{eq:plantModel} to the solution of the following optimization 
problem:
\begin{subequations}
\label{opt:objectiveproblemEquality}
\begin{align}
\label{opt:objectiveproblemEquality-a}
(u^*_t, y^*_t) := 
\underset{\bar u \in \real^m, ~ \bar y \in \real^p}{\arg \min}
& 
\phi_t(\bar u) + \psi_t(\bar y),\\
\label{opt:objectiveproblemEquality-b}
\text{s.t.} ~~~~~  & \bar y = G \bar u+Hw_t, \,\,\,\, 
K_t \bar y = e_t,
\end{align}
\end{subequations}
which contains only equality constraints on the system outputs.
In contrast with the method proposed in Section 
\ref{sec_projected_systems}, which guarantees tracking of an 
approximate optimizer, in this section we will show that, when the 
optimization problem includes only equality constraints, we can 
guarantee tracking of the exact optimizer (this behavior is achieved 
without resorting to a regularized Lagrangian).

\subsection{Controller Synthesis}
We begin by imposing  the following  assumption.
\vspace{.1cm}
\begin{assumption}
\label{ass:rankConditionConstraints}
The columns of $K_t G$ are linearly independent and there exists 
$\underline{k}, \bar k \in \real_{>0}$ such that
$\underline{k} I \preceq K_t G G^\tsp K_t^\tsp \preceq \bar k I$ for all $t$. \QEDB
\end{assumption}
\vspace{.1cm}
Since problem \eqref{opt:objectiveproblemEquality} contains 
only equality constraints, Assumption 
\ref{ass:rankConditionConstraints} is sufficient to
guarantee  uniqueness of the optimal multipliers~
\cite{qu2018exponential}. 
In what follows, for notation simplicity we will state the 
results by considering a time-invariant constraint matrix $K$. The 
stated results directly extend to the case of time-varying matrices, 
as noted in pertinent remarks.

We consider the following Lagrangian function 
for~\eqref{opt:objectiveproblemEquality}: 
\begin{align*}
\mc L_t(u,\lambda) = \phi_t(u) + \psi_t(Gu+Hw_t) + \lambda^\tsp(K(Gu+Hw_t)-e_t),
\end{align*}
where  $\lambda \in \realpos^r$ is the vector of dual 
variables. Under Assumptions 
\ref{ass:lipschitzAndConvexity} and 
\ref{ass:rankConditionConstraints},
the unique minimizer $(u^*_t,y^*_t)$ of 
\eqref{opt:objectiveproblemEquality} solves the following 
Karush–Kuhn–Tucker (KKT) conditions:
\begin{align}
0 &= \nabla \phi_t(u^*_t) + G^\tsp \nabla \psi_t(Gu^*_t+Hw_t) 
+ G^\tsp K^\tsp \lambda^*_t, \nonumber\\
0 &= K (Gu^*_t+Hw_t) - e_t.
\label{eq:KKTequalityConstrain}
\end{align}
To synthesize a controller, we define the following functions, which 
can be interpreted as modified gradients of the Lagrangian function:
\begin{subequations}
\begin{align}
L_{u,t}(u,y,\lambda) &:= \nabla \phi(u) + G^\tsp \nabla \psi(y) 
+ G^\tsp K^\tsp \lambda,\\
L_{\lambda,t}(y) &:= K y - e,
\end{align}
\end{subequations}
where (similarly to \eqref{eq:LuProjection}) with respect to the 
gradients of $\mc L_t(u,\lambda)$, the steady-state map $Gu+Hw_t$ has
been replaced by the variable $y$.
We then consider the following online primal-dual gradient 
controller applied to the plant \eqref{eq:plantModel}:
\begin{subequations}
\label{eq:primal-DualController}
\begin{align}
\label{eq:primal-DualController-a}
\varepsilon\dot x &= A x + B u + E w_t, ~~~~y= C x + D w_t,\\    
\label{eq:primal-DualController-b}
\dot u &= - \eta_u L_{u,t}(u,y,\lambda),\\
\label{eq:primal-DualController-c}
\dot \lambda &= \eta_\lambda L_{\lambda,t}(y),
\end{align}
\end{subequations}
where $\varepsilon, \eta_u, \eta_\lambda \in \real_{>0}$ are plant and 
controller gains. 
Similarly to the projected controller in Section 
\ref{sec_projected_systems},  the controller 
\eqref{eq:primal-DualController-b}--\eqref{eq:primal-DualController-c} 
uses output-feedback from the plant, and does not require any 
knowledge on $w_t$. In the following lemma, we relate the time-varying 
equilibria of \eqref{eq:primal-DualController} with the solution of 
\eqref{opt:objectiveproblemEquality}.
To this aim, in what follows we use the notation:
\begin{align}
\label{eq:z-zStar}
z &:= (u, \lambda), &
z^*_t &:= (u^*_t, \lambda^*_t), &
\tilde z &:= z - z^*_t,
\end{align}
to denote the controller state, the saddle-point of $\mc L_t(u,\lambda)$, and the controller tracking error, respectively.
Similarly, we use
\begin{align}
\label{eq:xi-xiStar}
\xi &:= (x, z), &
\xi^*_{t} &:= (x^*_{t}, z^*_{t}), &
\tilde{\xi} &=\xi-\xi_{\nu,t}^*,
\end{align}
to denote the joint state of \eqref{eq:primal-DualController}, the 
saddle-point of $\mc L_t(u,\lambda)$, with  
$x_{t}^*=-A^\inv(Bu_{\nu,t}^*+Hw_t)$, and the joint plant and 
controller tracking error, respectively.

\vspace{0.1cm}
\begin{lemma}
\label{lem:equilibriaNonProjected}
For any fixed $w_t \in \real^q$, let 
$\sbs{\xi}{eq} := (\sbs{x}{eq}, \sbs{u}{eq}, \sbs{\lambda}{eq})$
denote an equilibrium of \eqref{eq:primal-DualController}.
If Assumptions  
\ref{ass:stabilityPlant}--\ref{ass:rankConditionConstraints} hold, then
$\sbs{\xi}{eq}$ is unique and it coincides with  the unique solution of
\eqref{eq:KKTequalityConstrain}.
\end{lemma}
\vspace{0.1cm}

%


The proof of this claim is omitted due to space limitations. 
Differently from Lemma \ref{prop:equilibriaProjectedSystem} that 
guarantees equivalence between the equilibrium point of the controlled 
system and an approximate optimizer (defined as the saddle point of the
augmented Lagrangian), Lemma~\ref{lem:equilibriaNonProjected} 
establishes that the equilibrium point 
of~\eqref{eq:primal-DualController} coincides with the exact optimizer
(namely, the saddle point of the (non-augmented) Lagrangian).

\subsection{Stability and Tracking Analysis}
We now investigate the transient behavior of the controlled 
system~\eqref{eq:primal-DualController}. We begin by showing
that, when \eqref{eq:plantModel} is infinitely fast, 
\eqref{eq:primal-DualController} converges exponentially to the 
solution of~\eqref{opt:objectiveproblem}.
\vspace{.1cm}
\begin{proposition}
\label{prop:exponentialPrimalDual}
Let Assumptions 
\ref{ass:stabilityPlant}--\ref{ass:rankConditionConstraints} hold, let 
\begin{align}
\label{eq:Pz}
P_{z} := \begin{bmatrix}
\ell I & G^\tsp K^\tsp \\
K G & \ell \frac{\eta_u}{\eta_\lambda} I
\end{bmatrix},
\end{align}
where $\ell := \ell_u + \norm{G}^2 \ell_y$. If $\varepsilon=0$ and the 
controller parameters are such that 
$\eta_u > \frac{4 \bar k}{\ell \mu} \eta_\lambda$, 
then  for any $t_0 \in \realpos$:
\begin{align}
\norm{\tilde z(t)} &\leq \sqrt{\kappa} \norm{\tilde z(t_0)} 
e^{-\frac{1}{2} \rho_z (t-t_0)} 
+ \frac{ 4 \norm{P_z} \sqrt{\kappa}}{\underline 
\lambda (P_z)} \text{ess} \sup_{\tau \geq t_0} \norm{\dot z^*_\tau},
\label{eq:tracking-saddle-flow-equality}
\end{align}
for all $t \geq t_0$, 
$\rho_z := \frac{1}{2} 
\min \{\eta_\lambda {\underline{k}}/{\ell}, 
\eta_u \frac{\mu}{2} \}$, $\kappa = \bar \lambda(P_z)/\underline \lambda(P_z)$, where $\tilde z$ denotes the controller tracking 
error as in \eqref{eq:z-zStar}.
\end{proposition}
\smallskip

The proof of this result is presented in Appendix 
\ref{sec:proofsNonProjected}. 
Proposition \ref{prop:exponentialPrimalDual}
guarantees that \eqref{eq:primal-DualProjection} is input-to-state 
stable \cite{ES-YW:95} with respect to $\dot z^*_{\tau}$.
Two comments are in order. First, differently from 
\cite[Theorem 1]{qu2018exponential}, 
Proposition \ref{prop:exponentialPrimalDual} shows that $\rho_z$ can 
be made arbitrarily large by properly tuning the parameters $\eta_u$ 
and $\eta_\lambda$. Second, we note that the tracking result~\eqref{eq:tracking-saddle-flow-equality} is in the spirit of~\cite[Section~6]{cisneros2020distributed}; however, in~\cite{cisneros2020distributed} the primal-dual dynamics are assumed to be differentiable with respect to $t$ (in contrast, we require  milder conditions of absolute continuity).

\vspace{0.1cm}
\begin{remark}
\label{rem:timeVaryingK}
When the matrix $K$ is time-varying, then $P_{z}$ in~\eqref{eq:Pz} and the coefficient $\kappa$ in~\eqref{eq:tracking-saddle-flow-equality} are also time-varying. In 
this case, the result~\eqref{eq:tracking-saddle-flow-equality} extends 
by replacing $\kappa$ with $\sup_{\tau} \kappa_\tau$ and the 
coefficient $\frac{ 4 \norm{P_z} \sqrt{\kappa}}{\underline 
\lambda (P_z)}$ with
$\sup_\tau \frac{4 \norm{P_{z,\tau}} \sqrt{\kappa_\tau}}{\underline 
\lambda (P_{z,\tau})}$.
\QEDB
\end{remark}
\vspace{0.1cm}

We now present sufficient conditions on the time-scale separation 
between the plant and controller dynamics that result in exponential 
stability properties of the system~\eqref{eq:primal-DualController}. 

\vspace{0.1cm}
\begin{theorem}{\textit{(Stability and Tracking of~\eqref{eq:primal-DualController})}}
\label{thm:exponentialEqualityConstr}
Let Assumptions 
\ref{ass:stabilityPlant}--\ref{ass:rankConditionConstraints} hold and 
let $P_x, Q_x$ be as in Assumption \ref{ass:stabilityPlant}.
Suppose that $\varepsilon$ satisfies 
\begin{align}
\label{eq:lowerBoundRhoZ}
\varepsilon < 
\frac{\rho_z \underline \lambda(P_x) \underline \lambda(P_{z})}{
16 \sigma_1 \sigma_2 
+ 4 \rho_z \underline \lambda(P_{z}) \sigma_3},
\end{align}
where $P_z$, $\rho_z$ are as in Proposition 
\ref{prop:exponentialPrimalDual}, and
\begin{align*}
\sigma_1 &:= 2 \eta_u \ell_y \norm{C} \norm{G}(\ell + \norm{KG}) + 2 \eta_\lambda \norm{G^\tsp K^\tsp K C} + 2 \ell \eta_u \norm{KC}, \\
\sigma_2 &:= 2 \eta_u \ell \norm{P_x A^{-1} B} + 2 \eta_u \norm{P_x A^{-1} G G^\tsp K^\tsp},\\
\sigma_3 &:= 2 \eta_u \ell_y \norm{C} \norm{P_x A^{-1} B G^\tsp}.
\end{align*}
Then, for any $t_0 \in \realpos$, 
the tracking error \eqref{eq:xi-xiStar} satisfies:
\begin{align}
\norm{\tilde \xi(t)} \leq &
\sqrt{\kappa} \norm{\tilde \xi(t_0)} e^{-\frac{1}{2} \rho_\xi (t-t_0)} 
+ \frac{ 4 \norm{P_z} \sqrt{\kappa}}{
\rho_z \underline \lambda (P_z)} 
\text{ess} \sup_{\tau \geq t_0} \norm{\dot z^*_\tau} \nonumber\\
& 
+ \frac{4 \norm{P_x A^\inv E} \sqrt{\kappa}}{\underline \lambda(Q_x)} 
\text{ess} \sup_{\tau \geq t_0} \norm{\dot w_\tau},
\label{eq:iss-equality}
\end{align}
for all $t\geq t_0$, 
 $\kappa = \max \{\bar \lambda (P_x), \bar \lambda(P_z)\}/
\min \{\underline \lambda (P_x), \underline \lambda(P_z)\}$, 
\begin{align}
\rho_\xi = \frac{1}{4} \min \left\{
\rho_z \frac{\underline \lambda(P_z)}{\bar \lambda(P_z)},
\varepsilon^\inv \frac{\underline \lambda(Q_x)}{\bar \lambda(P_x)}
\right\}. 
\end{align}
\end{theorem}
\vspace{.1cm}
The proof of this result is postponed to Appendix 
\ref{sec:proofsNonProjected}.
Precisely, Theorem \ref{thm:exponentialEqualityConstr}
guarantees that \eqref{eq:primal-DualProjection} is input-to-state 
stable \cite{ES-YW:95} with respect to $\dot z^*_{\nu,\tau}$ and 
$\dot w_\tau$.
The bound on $\varepsilon$ is 
an increasing function of $\underline \lambda(P_x)$ and 
$\rho_z \underline \lambda(P_z)$, which are the convergence rates of 
the open-loop plant and of the controller with $\varepsilon=0$, 
respectively.
Moreover, we note that the rate of convergence $\rho_\xi$ is governed 
by the quantities $\rho_z$ and $\varepsilon$ (as well as matrices 
$P_x$, $Q_x$, and $P_z$), which are interpreted as the rates of 
convergence of the controller with $\varepsilon=0$ and the rate of 
convergence of the open-loop plant. 
Finally, we note that the bound~\eqref{eq:iss-equality} can be readily 
extended to account for time-varying matrices $K_t$ by adopting a 
reasoning similar to that in  Remark~\ref{rem:timeVaryingK}.

\begin{figure}[t]
\centering \subfigure[]{\includegraphics[width=.48\columnwidth]{%
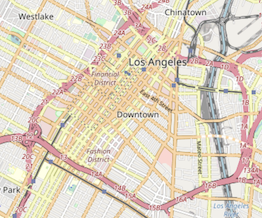}} 
\centering \subfigure[]{\includegraphics[width=.48\columnwidth]{%
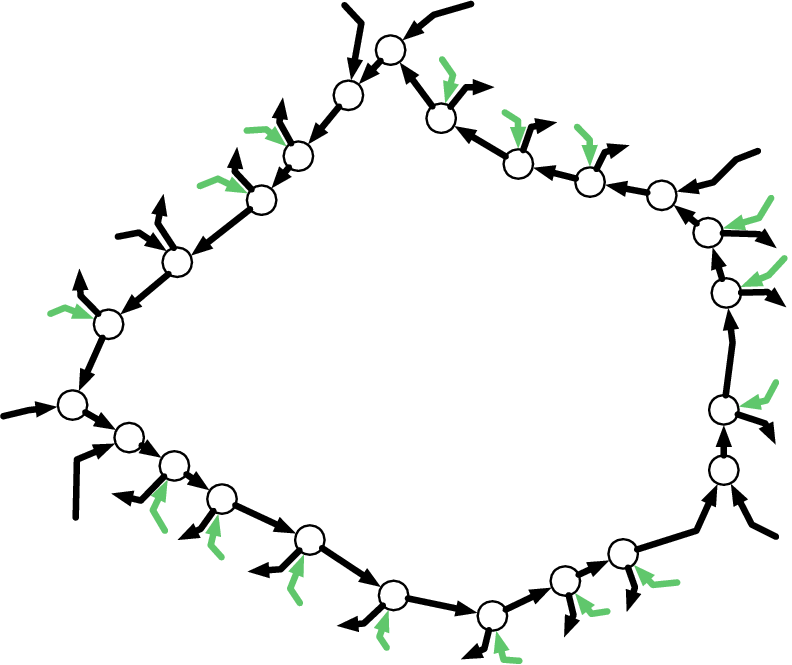}} 
\caption{(a) Portion of highway system in Los Angeles,~CA,~USA. 
(b) Network schematic. The network models 
$\vert \mc L \vert =64$ traffic highways and links colored in green 
represent controllable on-ramps.}
\label{fig:LAnetwor}
\vspace{-.5cm}
\end{figure}

\section{Application to Ramp Metering Control}
\label{sec_application_metering}
In this section, we apply the proposed framework to the control of 
on-ramps in a network of traffic highways\footnote{The code used
in our simulations is publicly available at 
\url{https://github.com/gianlucaBi/onlinePrimalDual_rampMetering}.}. 

To describe the traffic evolution, we adopt a continuous-time version 
of the Cell-Transmission Model (CTM) \cite{CFD:95}. 
We model a traffic network as a directed graph $\mc G=(\mc V, \mc L)$, 
where $\mc V$ models the set of traffic junctions (nodes) and  
$\mc L \subseteq \mc V \times \mc V$ models the set of highways 
(links). 
We partition the set of links into three disjoint sets:
$\mc L = \sbs{\mc L}{on} \cup \sbs{\mc L}{off} \cup \sbs{\mc L}{in}$, 
where $\sbs{\mc L}{on}$ denotes the set of on-ramps where vehicles can 
enter the network, $\sbs{\mc L}{off}$ denotes the set of off-ramps where 
vehicles  can exit the network, and $\sbs{\mc L}{in}$ denotes the set of 
internal links.

For $i \in \mc L$, we denote by $i^+$ the set of downstream links, and 
by $i^-$ the set of upstream links.
For all $i \in \mc L$, we let $\map{x_i}{\realpos}{\realpos}$ be the 
density of vehicle in the link. We model the dynamics of all links 
$i \in \sbs{\mc L}{in}$ according to the CTM with first-in-first-out 
(FIFO) allocation policy \cite{CFD:95}:
\begin{align}
\label{eq:CTM}
\dot x_i &= 
-\sps{f}{out}_i(x) + \sps{f}{in}_i(x),\nonumber\\
\sps{f}{out}_i(x) & = 
\min \{d_i(x_i), \{s_j(x_j)/r_{ij}\}_{j \in i^+}\}, \nonumber\\
d_i(x_i) &= \min \{ \varphi_i x_i, \sps{d}{max}_i\},
s_i(x_i) = \min \{ \beta_i (\sps{x}{jam}_i-x_i), \sps{s}{max}_i\},
\nonumber\\
\sps{f}{in}_i(x) &= \sum_{j \in i^-} \sps{f}{out}_j(x),
\end{align}
where $\map{d_i}{\realpos}{\realpos}$ and
$\map{s_i}{\realpos}{\realpos}$ are the link demand and supply 
functions, respectively, $r_{ij} \in [0,1]$ is the routing ratio
from $i$ to $j$, with $\sum_j r_{i j}=1$, $\varphi_i>0$. 
In our simulations, we used identical and uniform routing ratios 
at each junction. We refer to Fig. \ref{fig:LAnetwor} for an 
illustration of the network topology used in our simulations, and to 
Fig. \ref{fig:demandAndSupply} for a description of the parameters 
that characterize demand and supply.
For simplicity, all links are 
assumed to be identical.
The dynamics of on-ramps and off-ramps coincide with those of 
\eqref{eq:CTM}, where inflow and outflow functions are replaced by:
\begin{align}
\label{eq:flowsOnOffRamps}
\sps{f}{in}_i(x) &:= u_i, & & \text{if } i \in \sbs{\mc L}{on},\nonumber\\
\sps{f}{out}_i(x) &:= d_i(x_i), & & \text{if } i \in \sbs{\mc L}{off},
\end{align}

\begin{figure}[t]
\raisebox{-17mm}{
\centering \subfigure[]{\includegraphics[width=.4\columnwidth]{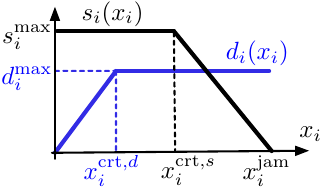}}}
\centering \subfigure[]{
\addstackgap[5pt]{ 		
\small{
\begin{tabular}{|c c c c|} 
\hline
Variable \!\!\!\!\!\! & Description & Value & Unit\\
 \hline\hline
$\varphi_i$ & free-flow speed & 4 & km/min\\ 
\hline
$\beta_i$ & back propag. speed & 4 & km/min\\ 
\hline
$\sps{d}{max}_i$ & demand saturation  & $120$ & veh/min\\ 
\hline
$\sps{s}{max}_i$ & supply saturation & $120$ & veh/h\\ 
\hline
$\sps{x}{jam}_i$ & jam density & $60$ & veh/km\\ 
\hline
$\sps{x}{crt,d}_i$ & critical density of $d_i$ & $30$ & veh/km\\ 
\hline
$\sps{x}{crt,s}_i$ & critical density of $s_i$ & $30$ & veh/km\\ 
\hline
-- & avg. numb. of lanes & $4$ & none\\ 
\hline
\end{tabular}
}
}
}
\caption{(a) Demand and supply functions. (b) Parameters description.}
\vspace{-.2cm}
\label{fig:demandAndSupply}
\end{figure}

We assume the availability of measurements that provide a noisy
estimate of the traffic densities in the highways: $y_i = x_i + w_i, \text{ for all } i \in \mc L$,
where $\map{w_i}{\realpos}{\real}$. Finally, we define the network 
throughput as the sum of all exit flows from the off-ramps $\Phi(x) := \sum_{i \in \sbs{\mc L}{off}} \sps{f}{out}_i(x)$. The  on-ramp metering problem is formalized as follows.
\begin{problem}{\textit{(Ramp Metering)}}
\label{prob:2-rampMetering}
Given a vector of on-ramp flow demands $\sbs{u}{ref} \in \real^m$, 
select the set of metered flows on the on-ramps  $(u_1, \dots, u_m)$ 
such that $u$ and $x$ minimize the cost $(u-\sps{u}{ref})^\tsp Q_u (u-\sps{u}{ref}) - \Phi(x)$, subject to the constraints \eqref{eq:CTM}-\eqref{eq:flowsOnOffRamps},
where $Q_u\in \real^{n\times n}$ is symmetric and positive definite.~
\QEDB
\end{problem}

We compare three control strategies, described next.

\label{sec_application_metering-B}
\subsubsection{Online Primal-Dual Controller}
To solve Problem \ref{prob:2-rampMetering}, we assume that 
for all $i \in \mc L$, the inequality
$\sps{d}{max}_i \leq \sps{s}{max}_j$ holds for all $j \in i^+$. 
Under this assumption, if the network is operated in a regime in 
which $x_i \leq \min\{\sps{x}{crt,$d$}_i, \sps{x}{crt,$s$}_i\}$ 
for all $i \in \mc L$ (i.e., all highways operate in the free-flow 
regime), then the dynamics \eqref{eq:CTM} simplify to the 
following linear model:
\begin{align}
\label{eq:CTM-linear}
\dot x_i &= 
-\sps{f}{out}_i(x) + \sps{f}{in}_i(x),\nonumber\\
\sps{f}{out}_i(x) & = \varphi_i x_i,  \quad\quad
\sps{f}{in}_i(x) = \sum_{j \in i^-} \sps{f}{out}_j(x).
\end{align}
In vector form, \eqref{eq:CTM-linear} can be written as
$\dot x = (R^\tsp-I) F x + Bu$, and 
$y = x + w$, where $R := [r_{ij}]$, and 
$F := \text{diag}(\varphi_1, \dots , \varphi_n)$.
Notice that matrix $(R^\tsp-I)F$ is Hurwitz (see e.g. 
\cite[Theorem 1]{GB-FP:19-tits}). Building on this, we propose the 
following problem:
\begin{align}
\label{opt:rampMetering2}
\min_{u,y} \;\;\; & (u-\sps{u}{ref})^\tsp Q_u (u-\sps{u}{ref}) 
- \Phi(y),\nonumber\\
\text{s.t.} \;\;\;  & y = -((R^\tsp -I)F)^\inv B u + w,\nonumber\\
& u_i \geq 0, \quad
y_i \leq \min\{\sps{x}{crt,$d$}_i, \sps{x}{crt,$s$}_i\}, 
\forall i \in \mc L.
\end{align}
The optimization problem \eqref{opt:rampMetering2} formalizes the 
objectives of the ramp metering problem, while guaranteeing that all 
highways are operated in the free-flow regime.

\subsubsection{Distributed Reactive Metering using ALINEA}
ALINEA \cite{MP-AK:02} is a distributed  metering strategy that 
has received considerable interest thanks to its simplicity of 
implementation and to its effectiveness.
Given a controllable on-ramp $i \in \sbs{\mc L}{in}$, ALINEA is a 
reactive
controller that takes the form 
$\dot u_i = \sum_{j \in i^+} K_j (\hat x_j - x_j)$, where 
$\hat x_i \in \realpos$ is a desired setpoint and $K_j$ are 
tunable controller gains. In our simulations, we let the setpoint be
$\hat x_i = \min\{\sps{x}{crt,$d$}_i, \sps{x}{crt,$s$}_i\}$.

\begin{figure}[t]
\centering \subfigure[]{\includegraphics[width=.5\columnwidth]{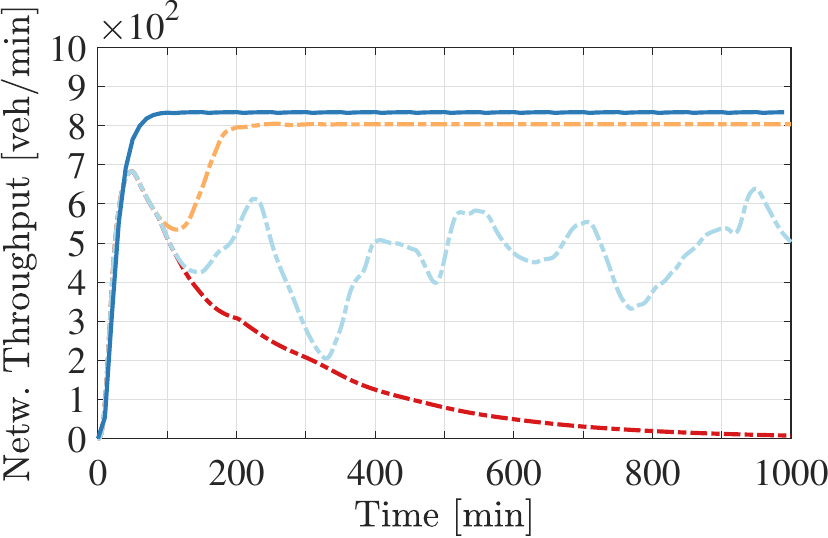}}%
\centering \subfigure[]{\includegraphics[width=.5\columnwidth]{%
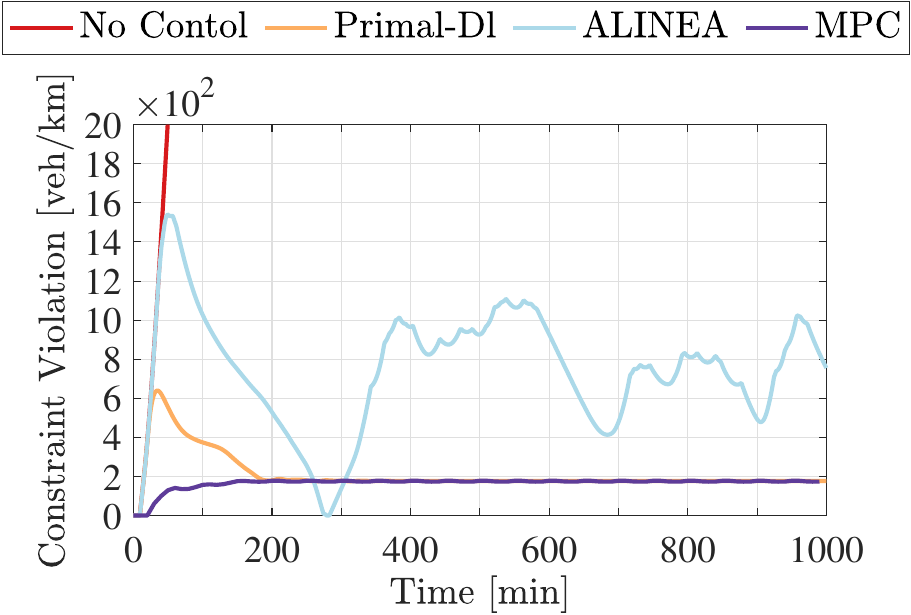}} 
\caption{\!\!\!\!
Plant without noise. 
(a) Network throughput $\Phi(x)$. 
(b) Constraint violation computed as 
$\Vert y - \min \{ \sps{x}{crt,$d$}, \sps{x}{crt,$s$}\}\Vert$.}
\label{fig:noiseFree}
\end{figure}

\begin{figure}[t]
\centering \subfigure[]{\includegraphics[width=.52\columnwidth]{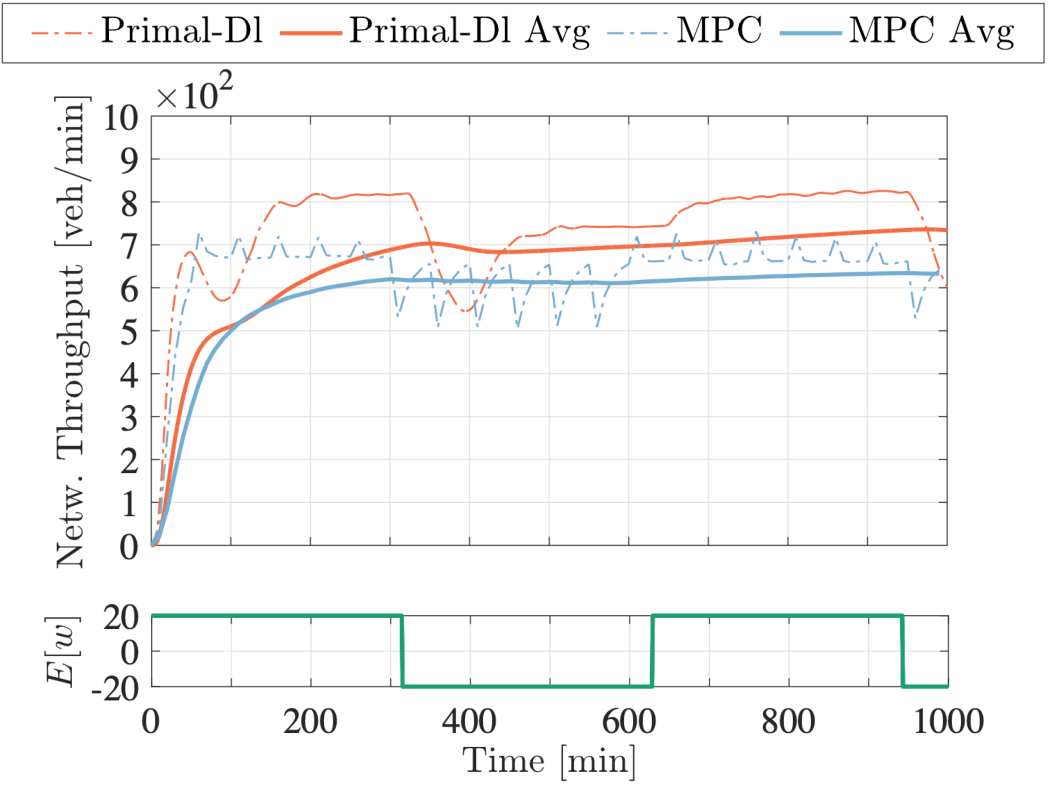}}%
\centering \subfigure[]{\includegraphics[width=.48\columnwidth]{%
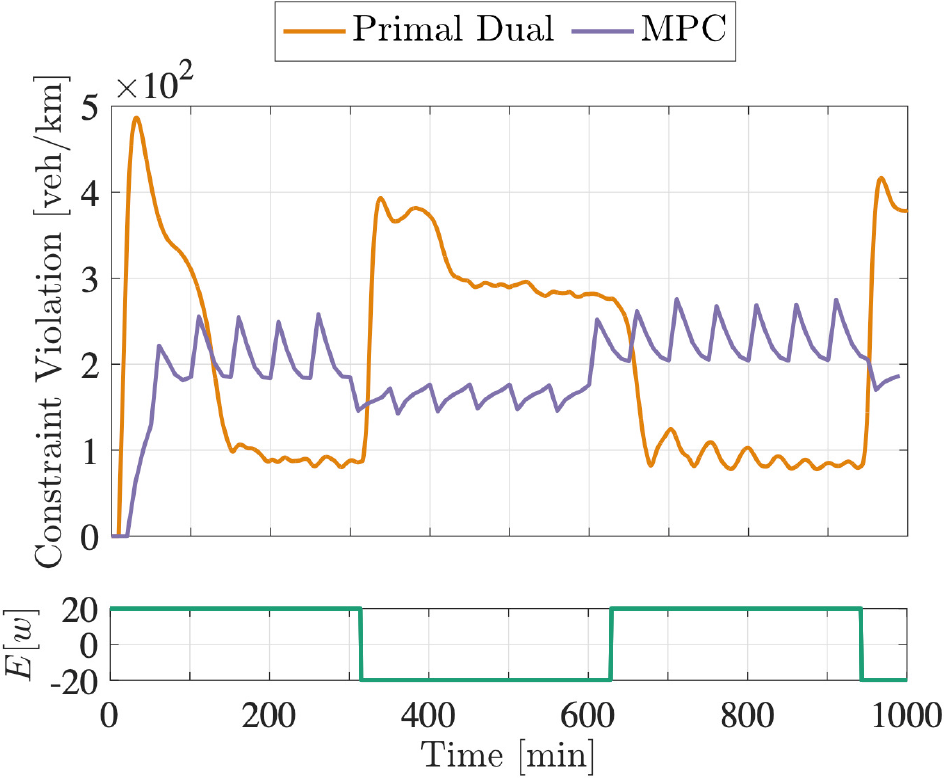}} 
\caption{Plant subject to random noise (green line shows noise 
mean). 
(a) Throughput $\Phi(x)$. 
(b) Constraint violation: 
$\Vert y - \min \{ \sps{x}{crt,$d$}, \sps{x}{crt,$s$}\}\Vert$.}
\label{fig:noise2}
\vspace{-.4cm}
\end{figure}

\subsubsection{Model Predictive Control (MPC)}
MPC is a receding-horizon control algorithm that computes an 
optimal control input based on a prediction of the system’s 
future trajectory according to the system’s dynamics.
We consider a formulation of MPC where the  optimization problem is 
solved every $T_s \in \real_{>0}$ time instants with prediction horizon
$T_p \in \real_{>0}$, with 
$T_p>T_s$.
In our simulations, we discretized the dynamics 
with $T_p=200$ min, $T_s=50$ min, and we used the cost 
function 
$\sum_{k=0}^{T_p}(u(k)-\sps{u}{ref})^\tsp Q_u (u(k)-\sps{u}{ref}) - \Phi(x(k))$.

\subsubsection*{Discussion}
Fig. \ref{fig:noiseFree} compares the performance of the three 
controllers in the noiseless case (i.e., where $w_i=0$ at all times 
for all $i \in \mc L$).
The simulation demonstrates that our method and MPC achieve the 
largest network throughput, outperforming ALINEA. Moreover, the 
constraint violation plot (right figure) shows that both our method 
and MPC are able to maintain the network in a regime
near the free-flow conditions.
Notice that, while for MPC this regime is precisely 
modeled through the prediction equations, the 
primal-dual controller maintains the system in such 
regime thanks to the constraints in \eqref{opt:rampMetering2}.
Finally, although ALINEA largely outperforms absence of on-ramp 
metering control, it critically suffers from its distributed 
architecture, making it suboptimal.

Fig. \ref{fig:noise2} compares the performance of our controller with 
that of  MPC in a scenario with time-varying output disturbance 
(depicted in green).
The simulation suggests that there are two main benefits in 
adopting primal-dual controllers as compared to MPC: (i) because the 
primal-dual controller uses 
instantaneous feedback from the system, it can react faster to 
unmodeled dynamics or time-varying disturbances, and (ii) in 
contrast with MPC where an optimization problem must be solved to 
convergence at the beginning of every time-window $[0,T_s]$, the 
primal-dual controller performs only one gradient-like step at every 
time.

\section{Conclusions}
\label{sec_conclusions}
We have leveraged online primal-dual dynamics  to develop an output controller that regulates an LTI plant to the solution of a time-varying optimization problem. For optimization problems with input constraints and output inequality constraints, we leveraged an augmented Lagrangian 
function and established exponential convergence to an approximate solution of the optimization problem.
For optimization problems with output equality constraints, we established exponential convergence to an interval around the exact optimal solution trajectory. Our convergence bounds capture the time-variability of the optimal  solution due to time-varying costs and constraints as well as the variation of the exogenous input. %

\counterwithin{theorem}{section}        
\appendices

\section{Analysis of Projected Saddle-Point Controller}
\label{sec:proofsProjected}
In this section, we present the proof of Proposition \ref{prop:exponentialProjectedPrimalDualDisturbance} and Theorem \ref{thm:exponentialInequalityConstr}. 
For the subsequent analysis, it is convenient to define the following time-varying map:
\begin{align}
\label{eq:saddle-point-map}
\hspace{-.2cm} F_t(z) := 
\left[ \hspace{-.1cm}
\begin{array}{c}
\nabla \phi_t(u) + G^\tsp \nabla \psi_t(Gu+Hw_t) + G^\tsp K_t^\tsp \lambda \\
     -\left(K_t(Gu+Hw_t) - e - \nu \lambda\right)
\end{array}
\hspace{-.1cm} \right] . \hspace{-.1cm}
\end{align}

\subsubsection{Proof of Proposition 
\ref{prop:exponentialProjectedPrimalDualDisturbance}} 
We consider only the case where the \text{ess-sup} in \eqref{eq:zBoundProjection} is bounded since otherwise the bound  holds trivially.
Recall that $z := (u, \lambda)$. We note that, when 
$\varepsilon=0$, the dynamics \eqref{eq:primal-DualProjection} can be 
rewritten as:
\begin{align}
\label{eq:projectedZmapping}
\dot z &= P_{\Omega} \big(z-\eta F_t(z)\big) - z,
\end{align}
where $\Omega := \mc U \times \mc C$. Proposition 
\ref{prop:exponentialProjectedPrimalDualDisturbance} leverages this 
structure as well as four auxiliary lemmas. 
The following lemma follows directly from 
\cite[Lemma 6]{NKD-SZK-MRJ:17} and \cite{JC-SN:19}.

\begin{lemma}
\label{lem:linearGradient}
Let Assumption \ref{ass:lipschitzAndConvexity} hold. 
Then, for any $t\geq0$, $u,  u' \in \real^m$ and $y,  y' \in \real^p$, 
there exist symmetric matrices 
$T_{u,t} \in \real^{m \times m}$ and 
$T_{y,t} \in \real^{p \times p}$, which satisfy
$\mu_u I \preceq T_{u,t}  \preceq \ell_u I$ and 
$0 \preceq T_{y,t} \preceq \ell_y I$, such that
$\nabla \phi_{t}(u) - \nabla \phi_{t}(u') = T_{u,t} (u-u')$ and 
$\nabla \psi_{t}(y) - \nabla \psi_{t}(y') = T_{y,t} (y-y')$. 
\end{lemma}
\vspace{.1cm}

Although the time-varying matrices $T_{u,t}$ and $T_{y,t}$ are functions of $u, u'$ and $y, y'$, respectively, this result allows us 
to leverage the relationships $\mu_u I \preceq T_{u,t} \preceq \ell_u I$ and  $0 \preceq T_{y,t} \preceq \ell_y I$.
Next, we show that $F_t(z)$ is strongly monotone and  globally Lipschitz continuous, uniformly in $t$.

\begin{lemma}\label{lem:strongMonotonicity}
Let Assumption \ref{ass:lipschitzAndConvexity} hold. 
Then, \eqref{eq:saddle-point-map} satisfies:
\begin{align}\label{eq:strongMonotoneMap}
(z-z')^\tsp (F_t(z)-F_t(z')) \geq \min \{ \mu_u, \nu\} \norm{z - z'}^2,
\end{align}
for all $z, z' \in \real^{m + r}$, and all $t \in \realpos$.
\end{lemma}
\begin{proof}
By expanding the left-hand side of \eqref{eq:strongMonotoneMap}, and by using Lemma \ref{lem:linearGradient}:
\begin{align*}
(z&-z')^\tsp (F_t(z)-F_t(z') )
= (u-u')^\tsp \big(\nabla \phi_t(u) -\nabla \phi_t(u')\big) \\
&\quad + (u-u')^\tsp G^\tsp (\nabla \psi_t(Gu+Hw_t) 
- \nabla \psi_t(Gu'+Hw_t)) \\
&\quad+ \nu \norm{\lambda - \lambda'}^2\\
&= (u-u')^\tsp (T_{u,t} + G^\tsp T_{y,t} G) (u-u')
+ \nu \norm{\lambda - \lambda'}^2\\
&\geq \mu_u \norm{u-u'}^2 + \nu \norm{\lambda - \lambda'}^2\geq \min \{ \mu_u, \nu \} \norm{z-z'}^2,
\end{align*}
which proves the claim.
\end{proof}

\begin{lemma}
\label{lem:lipschitzMapping}
Let Assumptions \ref{ass:lipschitzAndConvexity} and 
\ref{ass:slaterConditionAndRank} hold. Then, the mapping  \eqref{eq:saddle-point-map}
satisfies:
\begin{align}
\label{eq:lipschitzMap}
\norm{F_t(z) - F_t(z')} \leq \ell \norm{z-z'},
\end{align}
for all $z, z' \in \real^{m + r}$, and all $t\in \realpos$, where  
$\ell =: \sqrt{2} \max \{ \ell_u + \ell_y \norm{G}^2 + 
\norm{G}\overline{K}, \nu + \overline{K}\norm{G}\})$.
\end{lemma}
\begin{proof}
Using \eqref{eq:LuProjection}, we directly obtain the bounds:
\begin{align*}
&\norm{L_{u,t}(u,Gu+Hw,\lambda) - L_{u,t}(u',Gu'+Hw,\lambda')}\\
& \quad\quad\quad\quad\quad\quad
\leq (\ell_u + \ell_y \norm{G}^2) \norm{u-u'} + 
\overline{K} \norm{G} \norm{\lambda - \lambda'},\\
&\norm{L_{\lambda,t}(Gu+Hw,\lambda) 
- L_{\lambda,t}(Gu'+Hw,\lambda')}\\
&\quad\quad \quad\quad\quad\quad
\leq \overline{K}\norm{G} \norm{u-u'} +  \nu \norm{\lambda - \lambda'}.
\end{align*}
Finally, the claim follows by using the relationship:
$\norm{u-u'}+\norm{\lambda - \lambda'} \leq \sqrt{2} \norm{z-z'}$.
\end{proof}
\smallskip

The following result establishes that the existence of an ISS-Lyapunov
function with a particular structure guarantees input-to-state 
stability with exponential convergence rate, and it is a particular 
case of \cite[Ch.~4]{HKK:96} (see also \cite{edwards2000input}).


\smallskip
\begin{lemma}
\label{lemm:ISSbound}Consider the  system $\dot x=f(t,x,u)$, where 
$\map{f}{\realpos \times \real^n\times \real^m}{\real^n}$ is locally 
Lipschitz in $t$, $x$, and $u$, and $t \mapsto u(t)$ is 
measurable and essentially bounded. 
If there exists a smooth 
$\map{V}{\realpos \times \real^n}{\real}$ s.t.:
\begin{subequations}
\begin{align}
\label{eq:UGES-a}
\underline a \norm{x}^2 &\leq V(t,x) \leq \bar a \norm{x}^2,\\
\label{eq:UGES-b}
\frac{d}{dt}V(t,x) &\leq - b V(t,x), 
\quad \forall \norm{x} \geq b_0>0,
\end{align}
\end{subequations}
hold a.e., then, for all $t_0 \in \realpos$  and $x(t_0) \in \real^n$:
\begin{align}
\norm{x(t)} \leq 
\sqrt{\bar a/\underline a} (
\norm{x(t_0)} e^{- \frac{1}{2}  b (t-t_0)} + b_0), \quad \forall t\geq t_0.
\end{align}
\end{lemma}

\vspace{.2cm}

Using the results above, we now present the proof of 
Proposition~\ref{prop:exponentialProjectedPrimalDualDisturbance}. 
In particular, we show that the function 
$V(\tilde z_\nu) = \frac{1}{2} \norm{\tilde z_\nu}^2$ 
satisfies the assumptions of Lemma \ref{lemm:ISSbound}, where we 
recall that $\tilde z_\nu$ is as in \eqref{eq:z-zNuTilde}. 
In what follows, we let  $\hat z:= P_{\Omega} (z-\eta F_t(z))$.
By expanding the time-derivative:
\begin{align}
\label{eq:auxEq4}
\frac{d}{dt} V(\tilde z_\nu) &= 
- \tilde z_\nu^\tsp (z-\hat z) - \tilde z_\nu^\tsp \dot z^*_{\nu,t} , 
\end{align}
where we recall that $\dot z^*_{\nu,t}$ exists a.e. (see 
Remark~\ref{rem:absoluteContinuity}). 
Next, we recall that the projection operator is the unique vector 
$P_\Omega(z)$ that satisfies:
\begin{align}
\label{eq:projectionInequality}
(v' - P_\Omega(v))^\tsp(P_\Omega(v)-v) \geq 0, \text{ for all } 
v' \in \Omega.
\end{align}
By using \eqref{eq:projectionInequality} with $v' =z^*_{\nu,t}$ 
and $v = z-\eta F_t(z)$, we obtain the relationship 
$(\tilde z_\nu + \eta F_t(z))^\tsp (z - \hat z)  \geq 
\norm{z-\hat z}^2 + \eta (z -  z_{\nu,t}^*)^\tsp F_t(z)$.
Moreover, by recalling that  $\eta F_t(z)^\tsp(z-z^*) \geq 0$ 
(see Remark~\ref{rem:absoluteContinuity}), the first term in 
\eqref{eq:auxEq4} satisfies:
\begin{align}
\label{eq:auxBound}
- \tilde z_\nu^\tsp (z-\hat z) &\leq 
-\norm{z-\hat z}^2 - \eta (\hat z-z^*_{\nu,t})^\tsp F_t(z) \nonumber\\
& =
-\norm{z-\hat z}^2 
- \eta (\hat z - z^*_{\nu,t})^\tsp (F_t(z)  - F_t(z^*_{\nu,t}))\nonumber\\
&\quad\quad - \eta (\hat z - z^*_{\nu,t})^\tsp F_t(z^*_{\nu,t})\nonumber\\
& = 
-\norm{z-\hat z}^2 - \eta \tilde z_\nu^\tsp (F_t(z)-F_t(z^*_{\nu,t})) \nonumber\\
& \quad\quad+ \eta(z-\hat  z)^\tsp (F_t(z)-F_t(z^*_{\nu,t}))\nonumber\\
&\leq 
-\norm{z-\hat z}^2  
+ \eta \ell \norm{z-\hat z}\norm{\tilde z_\nu} 
- \eta \mu \norm{\tilde z_\nu}^2 \nonumber\\
&\leq 
- \eta \left( \mu - \eta \ell^2/4 \right) 
\norm{\tilde z_\nu}^2,
\end{align}
where the first equality follows by adding and subtracting 
$\eta (\hat z - z^*_{\nu,t})^\tsp F_t(z^*_{\nu,t})$, the second
equality follows by using $\hat z-z^*_{\nu,t} = (z-z^*_{\nu,t})-
(z - \hat  z)$ and by using
$(\hat z_\nu-z^*_{\nu,t})^\tsp F_t(z) \geq 0$, the fourth inequality follows 
from Lemmas \ref{lem:strongMonotonicity} and 
\ref{lem:lipschitzMapping}, 
and the last inequality follows by using  the relationship 
$2 ab \leq a^2 + b^2$ with
$a = \norm{z - \hat z}$ and 
$b = \frac{1}{2} \eta \ell \norm{\tilde z_\nu}$.
By substituting into \eqref{eq:auxEq4} we obtain:
\begin{align*}
\frac{d}{dt} V(\tilde z_\nu) & \leq
- \eta (\mu - \frac{\eta \ell^2}{4})\norm{\tilde z_\nu}^2 +
\norm{\tilde z_\nu} \norm{\dot z^*_{\nu,t}}\\
& \leq - \frac{\eta }{2} 
(\mu - \frac{\eta \ell^2}{4})\norm{\tilde z_\nu}^2,
\end{align*}
where the last inequality holds if $\norm{\tilde z_\nu} \geq \frac{2}{\eta(\mu - \eta \ell^2/4)} 
\text{ess} \sup_{\tau \geq t_0} \norm{\dot z^*_{\nu,\tau}}$. Finally, 
the  claim follows by application of Lemma \ref{lemm:ISSbound} with 
$\bar a = \underline a = \frac{1}{2}$,
$b = (\mu - \eta \ell^2/4)$, and 
$b_0 = \frac{2}{\eta(\mu - \eta \ell^2/4)} 
\text{ess} \sup_{\tau \geq t_0} \norm{\dot z^*_{\nu,\tau}}$. \hfill $\blacksquare$

\vspace{.2cm}
\subsubsection{Proof of Theorem \ref{thm:exponentialInequalityConstr}}
We consider only cases where the \text{ess-sup} in \eqref{eq:xiBoundProjection} are bounded, otherwise the bound holds trivially.
Our proof leverages singular perturbation arguments 
inspired by \cite[Ch. 11]{HKK:96}.
We first perform a change of variables for  
\eqref{eq:primal-DualProjection}.
Let $z := (u, \lambda)$, $\tilde x := x + A^\inv B u + A^\inv E w_t$,
and 
\begin{align*}
F_t(z, \tilde x) := \begin{bmatrix}
L_{u,t}(u, C \tilde x + Gu +Hw_t, \lambda)\\
L_{\lambda,t}(C \tilde x + Gu +Hw_t, \lambda)
\end{bmatrix}.
\end{align*}
Then, the dynamics \eqref{eq:primal-DualProjection} can be 
rewritten as:
\begin{align}
\label{eq:primalDualProjectNewVariab}
\varepsilon \dot{\tilde x} &= A \tilde x + \varepsilon A^\inv B S
\dot z + A^\inv E \dot w_t, \nonumber\\
\dot z &= P_\Omega(z - \eta F_t(z, \tilde x))-z,
\end{align}
where $S = [I_m, 0]$, and $\Omega = \mc U \times \mc C$.
Moreover, let
$b := \eta \norm{C} (\ell_y \norm{G} + \bar K)$, and 
$g := 2\sqrt{2} \norm{P A^\inv B} k_0$.  To prove the theorem's 
statement, we will show that
\begin{align}
\label{eq:Uaux2}
U(\tilde z_\nu, \tilde x) &:= (1-\theta) V(\tilde z_\nu) + \theta W(z),
\end{align}
where $V(\tilde z_\nu) = \frac{1}{2} \norm{z(t)-z^*_{\nu,t}}^2$, 
$W(z) = \tilde x^\tsp P_x \tilde x$, and $\theta = b/(b+g)$ 
satisfies the assumptions of Lemma \ref{lemm:ISSbound}.
We recall that $\tilde z_\nu := z-z^*_{\nu,t}$ and
$\hat z:= P_{\Omega} (z-\eta F_t(z,\tilde x))$. 
The time-derivative of $V(t,z)$ along the trajectory of
\eqref{eq:primalDualProjectNewVariab} reads:
\begin{align}
\label{eq:Vdotaux1}
\frac{d}{dt}V(\tilde z_\nu)
& = \tilde z_\nu^\tsp (\hat z-z) - \tilde z_\nu^\tsp \dot z^*_{\nu,t}
\end{align}
almost everywhere. The first term satisfies:
\begin{align*}
\tilde z_\nu^\tsp (\hat z-z) 
&=
\tilde z_\nu^\tsp (P_\Omega(z-\eta F_t(z, 0))-z) \nonumber\\
&\quad  + \tilde z_\nu^\tsp (P_\Omega(z-\eta F_t(z, \tilde x))
-P_\Omega(z-\eta F_t(z, 0)))\nonumber\\
&\leq
\tilde z_\nu^\tsp (P_\Omega(z-\eta F_t(z, 0))-z) \nonumber\\
&\quad + \eta \norm{\tilde z_\nu} \norm{F_t(z, \tilde x)- F_t(z, 0)}\\
& \leq - \eta (\mu - \frac{\eta \ell^2}{4}) \norm{\tilde z_\nu}^2
+ \eta \norm{\tilde z_\nu} \norm{F_t(z, \tilde x)- F_t(z, 0))},
\end{align*}
where the first inequality follows from the non-expansiveness of the 
projection operator, namely:
\begin{align*}
\tilde z_\nu^\tsp &(P_\Omega(z-\eta F_t(z, \tilde x)) 
-P_\Omega(z-\eta F_t(z, 0)))\\
&\quad\quad \leq
\norm{\tilde z_\nu}\norm{P_\Omega(z-\eta F_t(z, \tilde x)) 
-P_\Omega(z-\eta F_t(z, 0))}\\
&\quad\quad \leq
\norm{\tilde z_\nu}\norm{(z-\eta F_t(z, \tilde x)) 
-(z-\eta F_t(z, 0))},
\end{align*}
and the second inequality follows from 
\eqref{eq:auxBound}.
By expanding:
\begin{align*}
&\norm{F_t(z, \tilde x)- F_t(z, 0)} \nonumber\\
& \quad \leq 
\left\|
\begin{bmatrix}
G^\tsp(\nabla f_y(C \tilde x+Gu+Hw_t)-\nabla f_y(Gu+Hw_t))\nonumber\\
- K_t C \tilde x
\end{bmatrix}
\right\| \nonumber\\
& \quad \leq
\norm{C}(\ell_y \norm{G} +  \bar K) \norm{\tilde x}.
\end{align*}
Hence, by recalling the definition of $b$ and $\rho_z$, 
\eqref{eq:Vdotaux1} satisfies:
\begin{align}
\label{eq:Vdotaux}
\frac{d}{dt} V(\tilde z_\nu) &\leq 
- \rho_z \norm{\tilde z_\nu}^2 +
b \norm{\tilde x}\norm{\tilde z_\nu} + \norm{\tilde z_\nu} \norm{\dot z^*_{\nu,t}}
\nonumber\\
& \leq - \frac{\rho_z}{2} \norm{\tilde z_\nu}^2 
+ b \norm{\tilde x}\norm{\tilde z_\nu},
\end{align}
where the last inequality holds if 
$\norm{\tilde z_\nu} \geq  \frac{2}{\rho_z} \text{ess} \sup \norm{\dot z_{\nu,t}^*}$.
The time-derivative of $W(\tilde x)$ along the trajectories of 
\eqref{eq:primalDualProjectNewVariab}:
\begin{align}
\label{eq:bound3}
\frac{d}{dt} W(z) &= 
\varepsilon^\inv \tilde x^\tsp (A^\tsp P_x + P_x A) \tilde x 
\nonumber\\
& \quad\quad + 2 \tilde x^\tsp P_x A^\inv B S \dot z 
+ 2 \tilde x^\tsp P_x A^\inv E \dot w_t \nonumber\\
& \leq 
- \varepsilon^\inv \underline{\lambda}(Q_x) \norm{\tilde x}^2
+ 2 \norm{P_x A^\inv B} \norm{\tilde x} \norm{S \dot z}\nonumber\\
&\quad\quad 
+ 2 \norm{P_x A^\inv B} \norm{\tilde x} \norm{\dot w_t}.
\end{align}
By expanding the terms:
\begin{align*}
& \norm{S \dot z} 
= 
\norm{ S( P_\Omega(z - \eta F_t(z, \tilde x))-z)} \nonumber\\
&=
\norm{ S( P_\Omega(z - \eta F_t(z, \tilde x))-z 
- P_\Omega(z - \eta F_t(z^*_{\nu,t}, 0))+z^*_{\nu,t} )}
\nonumber\\
& \leq 
\eta \|L_{u,t}(u, C \tilde x +Gu+Hw_t,\lambda) \nonumber\\
&\quad\quad\quad - L_{u,t}(u^*, Gu^*+Hw_t,\lambda^*)\| + 2 
\norm{u-u^*}\nonumber\\
& \leq 
\sqrt{2} \max \{2+\eta(\ell_u+\ell_y \norm{G}^2), \norm{K_tG}\} 
\norm{\tilde z_\nu}\nonumber\\
&\quad\quad\quad+ \eta \ell_y \norm{C}\norm{G} \norm{\tilde x},
\end{align*}
where the first inequality follows from the non-expansiveness of the 
projection operator  and the second inequality follows from 
Assumption  \ref{ass:lipschitzAndConvexity}.
By recalling the definition of $g$, by letting 
$d=2 \eta \ell_y \norm{P A^\inv B} \norm{C}\norm{G}$, and by
substituting into \eqref{eq:bound3}: 
\begin{align}
\label{eq:Wdotaux}
\frac{d}{dt}  W(\tilde x) 
& \leq
- \varepsilon^\inv \underline{\lambda}(Q_x) \norm{\tilde x}^2
+ d\norm{\tilde x}^2
\nonumber\\
&\quad + g \norm{\tilde x } \norm{\tilde z_\nu} 
+ 2 \norm{P_x A^\inv B} \norm{\tilde x} \norm{\dot w_t}\nonumber\\
& \leq
- \frac{\underline{\lambda}(Q_x)}{2 \varepsilon} \norm{\tilde x}^2
+ d\norm{\tilde x}^2 + g \norm{\tilde x } \norm{\tilde z_\nu},
\end{align}
where the last inequality is satisfied if
$\norm{\tilde x} \geq 
\frac{4 \varepsilon \norm{P_x A^\inv E}}{\underline \lambda(Q_x)} 
\text{ess} \sup \norm{\dot w_t}$.
By combining \eqref{eq:Vdotaux}-\eqref{eq:Wdotaux}:
\begin{align*}
\frac{d}{dt} U (\tilde x, \tilde z_\nu) \leq - \hat \xi^\tsp  \Lambda \hat \xi -
\frac{1}{2} \min \{ 2\rho_z, \frac{\underline \lambda(Q_x)}{2 \varepsilon \bar \lambda(P_x)}\},
\end{align*}
where
\begin{align*}
\Lambda := 
\begin{bmatrix}
(1-\theta) \frac{\rho_z}{4} & -\frac{1}{2}((1-\theta)b+\theta g)\\
\frac{1}{2}((1-\theta)b+\theta g) & \theta(\frac{\underline \lambda(Q_x)}{4\varepsilon}-d)
\end{bmatrix}.
\end{align*}
$\Lambda$ is positive definite when
$\theta (1-\theta) \frac{\rho_z}{4}(\frac{\underline \lambda(Q_x)}{4\varepsilon}-d)
> \frac{1}{4}((1-\theta) b + \theta g)^2$,
which holds when \eqref{eq:boundEpsilonInequality} is satisfied.
Finally, the claim follows by application of Lemma \ref{lemm:ISSbound}
with $\bar a = \max \{\frac{1}{2}, \bar \lambda (P_x)\}$, 
$ \underline a = \min \{\frac{1}{2}, \underline \lambda (P_x)\}$,
$c_3 = \frac{1}{2} \min\{ 2\rho_z, \frac{\underline{\lambda}(Q_x)}{4\varepsilon \bar \lambda(P_x)}\}$, 
and 
$b_0 = \max \{ 
\frac{2}{\rho_z}  \text{ess} \sup \norm{\dot z^*_{\nu,t}},
\frac{4 \varepsilon \norm{PA^\inv E} }{\underline \lambda(Q_x)} 
\text{ess} \sup \norm{\dot w_t}\}$.~
\QEDBL

\section{Analysis of Primal-Dual Controller}
\label{sec:proofsNonProjected}
In this section, we prove Proposition 
\ref{prop:exponentialPrimalDual} and Theorem 
\ref{thm:exponentialEqualityConstr}. We introduce the 
following  change of variables for \eqref{eq:primal-DualController}:
\begin{align*}
\tilde x &:= x - h(u,w_t), & h(u,w_t) &:= - A^\inv B u - A^\inv E w_t.
\end{align*}
The dynamics 
\eqref{eq:primal-DualController} are re-written in the new variables next.
\begin{lemma}
\label{lem:linearDynamics}
Let Assumption 
\ref{ass:stabilityPlant}-\ref{ass:rankConditionConstraints} be 
satisfied, and 
for any $t \in \realpos$, let $(u^*_t, \lambda_t^*)$ be the 
saddle-point of \eqref{opt:objectiveproblemEquality}.
The dynamics \eqref{eq:primal-DualController} have the following
equivalent representation:
\begin{align}
\label{eq:linearDynamics}
\varepsilon \dot {\tilde x} &= F_{11} \tilde x +F_{12} (u-u_t^*)
+ F_{13} (\lambda -\lambda_t^*) + F_{14} \dot w_t, \nonumber\\
\dot {u} &= F_{21} \tilde x  + F_{22} (u-u_t^*)
+ F_{23} (\lambda -\lambda_t^*), \nonumber\\
\dot {\lambda} &= F_{31} \tilde x  + F_{32}  (u-u_t^*),
\end{align}
where $F_{14} = \varepsilon A^\inv E$, 
\begin{align*}
F_{11} &= A - \varepsilon\eta_u A^\inv B G^\tsp T_{y,t} C,  &
F_{21} &= - \eta_u G^\tsp T_{y,t} C,\\
F_{12} &=  - \varepsilon\eta_u A^\inv B(T_{u,t} + G^\tsp T_{y,t} G),&
F_{23} &= - \eta_u G^\tsp K^\tsp,\\
F_{13} &= - \varepsilon\eta_u A^\inv B G^\tsp K^\tsp,&
F_{31} &= \eta_\lambda K C,\\
F_{22} &= - \eta_u (T_{u,t} + G^\tsp T_{y,t} G),&
F_{32} &= \eta_\lambda K G,
\end{align*}
and $T_{u,t}$, $T_{y,t}$ are symmetric 
matrices that satisfy $\mu_u I \preceq T_{u,t} \preceq \ell_u I$, 
$0 \preceq T_{y,t} \preceq \ell_y I$ uniformly in $t$.
\end{lemma}

\begin{proof}
By application of Lemma \ref{lem:linearGradient}:  
\begin{align*}
\dot {u} &= - \eta_u L_{u,t}(u,y,\lambda) +
\underbrace{\eta_u L_{u,t}(u^*_t,Gu_t^*+Hw_t,\lambda_t^*)}_{=0}\\
&= -\eta_u ((T_{u,t} + G^\tsp T_{y,t} G) (u-u_t^*) \\
& \quad\quad \quad\quad 
+ G^\tsp T_{y,t} C \tilde x
+ G^\tsp K^\tsp (\lambda-\lambda_t^*)),\\
\dot \lambda &= \eta_\lambda  L_{\lambda,t}(u,y,\lambda) - 
\underbrace{\eta_\lambda \nabla_\lambda  L_{\lambda,t}(u^*,Gu_t^*+Hw_t,\lambda_t^*)}_{=0}\\
&= \eta_\lambda (KC\tilde x + K G (u-u_t^*)).
\end{align*}
Finally, by using the relationships
$\varepsilon \dot {\tilde x} =  \dot x 
- \varepsilon \frac{\partial h}{\partial u} \dot u
- \varepsilon \frac{\partial h}{\partial w} \dot w_t$,  and by 
substituting the expression for $\dot {u}$:
\begin{align*}
\varepsilon \dot {\tilde x} &= A \tilde x 
+ \varepsilon A^\inv B \dot {u }
+ \varepsilon A^\inv E \dot w_t\\
&= (A - \varepsilon \eta_u A^\inv B G^\tsp T_{y,t} C) \tilde x 
- \varepsilon\eta_u A^\inv B G^\tsp K^\tsp  (\lambda-\lambda^*_t)\\
&\quad - \varepsilon\eta_u A^\inv B(T_{u,t} + G^\tsp T_{y,t} G) (u-u^*_t) 
+ \varepsilon A^\inv E \dot w_t,
\end{align*}
which proves the claim.
\end{proof}

\subsubsection{Proof of Proposition \ref{prop:exponentialPrimalDual}}
The proof follows similar ideas as \cite[Lemma 2]{qu2018exponential}. 
By letting $\varepsilon=0$ in \eqref{eq:linearDynamics} we obtain 
$A \tilde x=0$, which, by Assumption 
\ref{ass:stabilityPlant} implies $\tilde x =0$. 
Hence, we let $z := (u, \lambda)$ and $\tilde z := z - z_t^*$, and we 
rewrite the dynamics \eqref{eq:linearDynamics} as
$\dot{z} =  F_z (z-z_t^*) = F_z \tilde z$, where 
\begin{align}
\label{eq:Fz}
F_z = \begin{bmatrix} F_{22} & F_{23}\\ F_{32} & 0\end{bmatrix}.
\end{align}

We will prove that 
$V(z) = {\tilde z}^\tsp P_z \tilde z$ satisfies the assumptions of Lemma \ref{lemm:ISSbound}.
By the Schur Complement, $P_z$ is positive definite if 
and only if 
$\ell^2 \frac{\eta_u}{\eta_\lambda} I - G^\tsp K^\tsp K G 
\succ 0$.
Using $\eta_u > \frac{4 \bar k}{\ell \mu} \eta_\lambda$, 
$\ell \geq \mu$ and Assumption \ref{ass:rankConditionConstraints} one gets $\ell^2 \frac{\eta_u}{\eta_\lambda} I - G^\tsp K^\tsp K G 
\succeq 
((4 \ell \bar k)/\mu_u) I - \bar k I 
\succeq 3  \bar k \succ 0$, which shows that $P_z$ is positive definite.
By expanding the time-derivative:
\begin{align}
\label{eq:auxVdotW5}
\frac{d}{dt} V(\tilde z) &= 
(\dot z - \dot z_t^*)^\tsp P_z (z - z_t^*)
+  (z - z_t^*)^\tsp P_z (\dot z - \dot z_t^*)\nonumber\\
&= \tilde z^\tsp (F_z^\tsp P_z + P_z F_z)  \tilde z
- 2 \tilde z^\tsp P_z \dot z_t^*.
\end{align}
Next, we show that 
$\tilde z^\tsp (F_z^\tsp P_z + P_z F_z)  \tilde z 
+ \bar \rho_z  V(\tilde z) \leq 0$,
where $\bar \rho_z = \min \{\eta_\lambda \frac{\underline{k}}{\ell}, 
\eta_u \frac{\mu}{2} \}$.
Let $M:= F_z^\tsp P_z +  P_z F_z + \bar\rho_z P_z$.
By expanding the product, $M = [M_{ij}]$ is a $2 \times 2$ 
block symmetric matrix with~blocks:
\begin{align}
\label{eq:Mblocks}
M_{11} &=  2 \eta_u \ell (T_{u,t}+G^\tsp T_{y,t} G) 
- 2 \eta_\lambda G^\tsp K^\tsp K G - \bar \rho_z \ell I,\nonumber\\
M_{12} &=  \eta_u (T_{u,t}+G^\tsp T_{y,t} G)^\tsp G^\tsp K^\tsp 
- \bar \rho_z G^\tsp K^\tsp,\nonumber\\
M_{22} &= 2 \eta_u K G G^\tsp K^\tsp 
- \bar \rho_z \ell (\eta_u/\eta_\lambda) I,
\end{align}
and $M_{21}=M_{12}^\tsp$.
By application of the Schur Complement, $M$ is positive 
definite when $M_{22} \succ 0$ and 
$M_{11}-M_{12} M_{22}^\inv M_{12}^\tsp \succ 0$.
The first condition can be rewritten as:
$M_{22} 
\succeq 
(2 \eta_u \underline{k} - \bar \rho_z \ell \frac{\eta_u}{\eta_\lambda}) I
\succeq
\eta_u \underline k I 
\succ 0$, 
where we used Assumption \ref{ass:rankConditionConstraints} and the 
expression of $\rho_z$.
For the second condition, we have:
\begin{align*}
M_{12} & M_{22}^\inv M_{12}^\tsp 
\preceq
M_{12} (\eta_u KG G^\tsp K^\tsp)^\inv M_{12}^\tsp \\
&= 
\eta_u (T_{u,t}+G^\tsp T_{y,t} G)^\tsp (T_{u,t}+G^\tsp T_{y,t} G) 
+ \frac{\bar \rho_z^2}{\eta_u}  I\\
&\quad\quad - \bar \rho_z ((T_{u,t}+G^\tsp T_{y,t} G)^\tsp + (T_{u,t}+G^\tsp T_{y,t} G))\\
& \preceq \eta_u \ell (T_{u,t}+G^\tsp T_{y,t} G)
+ \frac{\bar \rho_z^2}{\eta_u} I - 2 \bar \rho_z (T_{u,t}+G^\tsp T_{y,t} G),
\end{align*}
where the first bound follows from Assumption 
\ref{ass:rankConditionConstraints} and the definition of 
$\bar \rho_z$,
the second identity follows from 
$G^\tsp K^\tsp (K G G^\tsp K^\tsp)^\inv K G = I$, and the
last bound follows from $G^\tsp T_{y,t} G \succeq 0$.
Thus:
\begin{small}
\begin{align*}
&M_{11}- M_{12} M_{22}^\inv M_{12}^\tsp \succeq 
2 \eta_u \ell (T_{u,t}+ G^\tsp T_{y,t} G) 
- 2 \eta_\lambda G^\tsp K^\tsp K G \\
& \quad - \bar \rho_z \ell I 
- \eta_u \ell (T_{u,t}+ G^\tsp T_{y,t} G)
- \frac{\bar \rho_z^2}{\eta_u} I + 2 \bar \rho_z (T_{u,t}+G^\tsp T_{y,t} G),
\end{align*}
\end{small}
and, by using 
\begin{align*}
\frac{1}{2} \eta_u \ell &(T_{u,t}+G^\tsp T_{y,t} G) 
- 2 \eta_\lambda G^\tsp K^\tsp KG\\
&\succeq (\frac{1}{2} \eta_u \ell \mu_u
- 2 \eta_\lambda \bar k) I \succ 0    \\
\frac{1}{2} \eta_u \ell &(T_{u,t}+G^\tsp T_{y,t} G) - \rho \ell I \succeq (\frac{1}{2} \eta_u \ell \mu
- \rho \ell) I \succeq 0\\
\eta_u \ell &(T_{u,t}+G^\tsp T_{y,t} G) - \eta_u \ell (T_{u,t}+G^\tsp T_{y,t} G) =0,
\end{align*}
we conclude $M_{11}-M_{12} M_{22}^\inv M_{12}^\tsp \succ 0$,
which shows $M \succ 0$.

As a result, \eqref{eq:auxVdotW5} satisfies:
\begin{align}
\label{eq:VboundAux3}
\frac{d}{dt} V(\tilde z) 
&\leq 
- \bar \rho_z V(\tilde z)
+ 2 \norm{\tilde z} \norm{P_z} \norm{\dot z_t^*}\nonumber\\
&= 
- \frac{\bar \rho_z}{2} V(\tilde z)
- \frac{\rho_z}{2} \underline \lambda(P_z) \norm{\tilde z}^2 
+ 2 \norm{\tilde z} \norm{P_z} \norm{\dot z_t^*}\nonumber\\
& \leq 
- \frac{\bar \rho_z}{2}  V(\tilde z),
\end{align}
where the last inequality holds when
$2 \norm{\tilde z} \norm{P_z} \norm{\dot z_t^*} 
- \frac{\rho_z}{2} \underline \lambda(P_z) \norm{\tilde z}^2 \leq 0$,
or 
$\norm{\tilde z} \geq \frac{4 \norm{P_z}}{\rho_z \underline \lambda(P_z)}
\text{ess} \sup_\tau \norm{\dot z^*_\tau}$.
Finally, the claim follows by application of Lemma \ref{lemm:ISSbound}
with $\bar a = \bar \lambda(P_z),
\underline a = \underline \lambda(P_z)$, 
$b = \frac{\bar \rho_z}{2}$,  and
$b_0 = \frac{4 \norm{P_z}}{\rho_z \underline \lambda(P_z)}
\text{ess} \sup_\tau \norm{\dot z_\tau^*}$. 
\QEDBL

\smallskip
\subsubsection{Proof of Theorem \ref{thm:exponentialEqualityConstr}}
Our proof technique leverages singular perturbation arguments 
inspired by~\cite[Ch. 11]{HKK:96}.
Let $z := (u, \lambda)$, $\tilde z := z - z_t^*$ and rewrite the 
dynamics \eqref{eq:linearDynamics} as:
\begin{align}
\label{eq:auxDynamicsWithW}
\dot{ \tilde x} = F_{11} \tilde x +F_{xz} \tilde z + F_{14} \dot w_t, \,\,\,\, \dot{ \tilde z} = F_{zx} \tilde x + F_z \tilde z, 
\end{align}
where $F_z$ is as defined by  \eqref{eq:Fz}, 
$F_{xz} = [F_{12}, F_{13}]$, and
$F_{zx} = [F_{21}^\tsp, F_{31}^\tsp]^\tsp$. 
To show this claim, we will prove that the function 
$U(\tilde x,\tilde z) = (1-\theta) V(\tilde z) + \theta W(\tilde x)$, 
where $\theta = \norm{\sigma_1}/(\norm{\sigma_2} + 
\norm{\sigma_1})$ satisfies the assumptions of Lemma 
\ref{lemm:ISSbound}.
By substituting \eqref{eq:auxDynamicsWithW} and by using 
$F_z^\tsp P_z + P_z F_z \preceq - \bar \rho_z P_z$ 
(see \eqref{eq:auxVdotW5} and \eqref{eq:VboundAux3}):
\begin{align*}
\dot V(\tilde z) 
&= 
\tilde z^\tsp (F_z^\tsp P_z + P_z F_z) \tilde z 
+ 2 \tilde x^\tsp F_{zx} P_z \tilde z  
- 2 \tilde z^\tsp P_z \dot z^*\\
& \leq 
- \rho_z \tilde z^\tsp P_z \tilde z 
+ \tilde z^\tsp \sigma_1 \tilde x
- 2 \tilde z^\tsp P_z \dot z^*\\
& \leq 
- \frac{\rho_z}{2} \underline \lambda( P_z)  \norm{\tilde z}^2
+ \norm{\sigma_1} \norm{\tilde z} \norm{\tilde x},
\end{align*}
the last inequality holds when
$\norm{\tilde z} \geq \frac{4 \norm{P_z}}{\rho_z \underline 
\lambda(P_z)} \sup_\tau \norm{\dot z^*_\tau}$.
Next, by expanding the time-derivative of $W(\tilde x)$:
\begin{align*}
\varepsilon \dot W(\tilde x) 
&= 
\tilde x^\tsp (F_{11}^\tsp P_x + P_x F_{11}) \tilde x
+ 2 \tilde x^\tsp P_x F_{xz} \tilde z 
+ 2 \tilde x^\tsp P_x F_{14} \dot w_t.
\end{align*}
Using $F_{11} = A - \varepsilon\eta_u A^\inv B G^\tsp T_{y,t} C$, 
$A^\tsp P_x + P_x A = - Q_x$:
\begin{align*}
\tilde x^\tsp &(F_{11}^\tsp P_x + P_x F_{11}) \tilde x 
= 
- \tilde x^\tsp Q_x \tilde x \\
& - \eta_u \varepsilon \tilde x^\tsp (C^\tsp T_{y,t} G B^\tsp 
A^{\negat \tsp} P_x + P_x A^\inv B G^\tsp T_{y,t} C) \tilde x.
\end{align*}
Let $\Sigma_1 := 2 P_z [F_{21}^\tsp, F_{31}^\tsp]^\tsp$,
$\Sigma_2 := 2 \varepsilon^\inv P_x [F_{12}, F_{13}]$,  
$\Sigma_3 = \eta_u (C^\tsp T_{y,t} G B^\tsp 
A^{\negat \tsp} P_x + P_x A^\inv B G^\tsp T_{y,t} C)$, and $\Sigma_4 = P_x A^\inv E$. Then,
\begin{align}
\label{eq:myAuxW}
\varepsilon \dot W(\tilde x) 
& \leq
- \underline \lambda(Q_x) \norm{\tilde x }^2
+ \varepsilon \norm{\sigma_2} \norm{\tilde x} \norm{\tilde z}\\
& \quad\quad + \varepsilon \norm{\Sigma_3} \norm{\tilde x}^2
+ 2 \norm{\Sigma_4} \norm{\tilde x} \norm{\dot w_t}
\nonumber \\
& \leq
- \frac{\underline \lambda(Q_x)}{2} \norm{\tilde x }^2
+ \varepsilon \norm{\Sigma_2} \norm{\tilde x} \norm{\tilde z}
+ \varepsilon \norm{\Sigma_3} \norm{\tilde x}^2, \nonumber 
\end{align}
where the last inequality holds if
$- \frac{\underline \lambda(Q_x)}{2} \norm{\tilde x }^2 
+ 2 \norm{\Sigma_4} \norm{\tilde x} \norm{\dot w_t} \leq 0$,
or $\norm{\tilde x} \geq 
\frac{4 \norm{\Sigma_4}}{\underline \lambda(Q_x)} 
\text{ess} \sup_{\tau \geq 0} \norm{\dot w_\tau}$.
By using $V(z) \leq \bar \lambda(P_z) \norm{\tilde z}^2$,
$W(z) \leq \bar \lambda(P_x) \norm{\tilde x}^2$, 
by letting $\hat \xi :=(\norm{\tilde z}, \norm{\tilde x})$, and by
combining \eqref{eq:VboundAux3}-\eqref{eq:myAuxW} we get $\dot U (\tilde x, \tilde z) \leq - \hat \xi^\tsp  \Lambda \hat \xi -
\rho_\xi U (\tilde x, \tilde z)$,
where:
\begin{align*}
\Lambda \!\!=\!\! \begin{bmatrix}
(1-\theta)\frac{\rho_z \underline \lambda(P_z)}{4}  
& - \frac{1}{2}((1-\theta) \norm{\Sigma_1} + \theta \norm{\Sigma_2})\\
- \frac{1}{2}((1-\theta) \norm{\Sigma_1} + \theta \norm{\Sigma_2}) 
\!\!\!\!\!\!\!
& \theta (\frac{\underline \lambda (Q_x)}{4 \varepsilon}  - 
\norm{\Sigma_3})
\end{bmatrix}.
\end{align*}
Matrix $\Lambda$ is positive definite when 
\begin{align*}
\theta (1-\theta)
\frac{\rho_z \underline \lambda(P_z)}{4}
(\frac{\underline \lambda (Q_x)}{4 \varepsilon}  \!- \!\norm{\Sigma_3})
&> \frac{1}{4}((1-\theta) \norm{\Sigma_1} \!+\! \theta \norm{\Sigma_2})^2,
\end{align*}
which holds when the following is satisfied: 
\begin{align*}
\varepsilon <
\frac{\rho_z \underline \lambda(P_x) \underline \lambda(P_z)}{
16 \norm{\Sigma_1} \norm{\Sigma_2} 
+ 4 \rho_z \underline \lambda(P_z)\norm{\Sigma_3}}.
\end{align*}
The bound~\eqref{eq:lowerBoundRhoZ} is then obtained using standard manipulations.  Finally, the claim follows by application of Lemma \ref{lemm:ISSbound}
with 
$\bar a = \max \{\bar \lambda (P_x), \bar \lambda(P_z)\}$, 
$\underline a = \min \{\underline \lambda (P_x), \underline 
\lambda(P_z)\}$,
$b = \frac{1}{4} \min \left\{
\rho_z \frac{\underline \lambda(P_z)}{\bar \lambda(P_z)},
\varepsilon^\inv \frac{\underline \lambda(Q_x)}{\bar \lambda(P_x)}
\right\}$, and 
$b_0 = \max \{ 
\frac{4 \norm{P_z}}{\rho_z \underline \lambda(P_z)} 
\text{ess} \sup_\tau \norm{\dot z_\tau^*},
\frac{4 \norm{\Sigma_4}}{\underline \lambda(Q_x)} 
\text{ess} \sup_\tau \norm{\dot w_\tau}
\}$.
\QEDBL
%

\section*{Acknowledgments}
The authors would like to thank the anonymous Reviewers and the  
Associate Editor for the constructive review comments.

\bibliographystyle{IEEEtran}
\bibliography{bib/alias,bib/bibliography,bib/bibliography2,bib/GB,bib/New}

\end{document}